\documentclass[12pt]{article}
\usepackage{amssymb, amsmath, amsthm, amscd}
\usepackage[pdftex]{graphics}
\usepackage[utf8]{inputenc}
\usepackage[all,cmtip]{xy}
\usepackage{enumitem}
\usepackage{setspace}
\usepackage{mathdots}

\usepackage[mathscr]{euscript}
\usepackage[colorlinks=true]{hyperref}
\hypersetup{colorlinks   = true}
\hypersetup{linkcolor=blue}
\usepackage{tikz}
\usetikzlibrary{cd}

\begin{document}

\newtheorem{The}{Theorem}[section]
\newtheorem{Lem}[The]{Lemma}
\newtheorem{Prop}[The]{Proposition}
\newtheorem{Cor}[The]{Corollary}
\newtheorem{Rem}[The]{Remark}
\newtheorem{Obs}[The]{Observation}
\newtheorem{SConj}[The]{Standard Conjecture}
\newtheorem{Titre}[The]{\!\!\!\! }
\newtheorem{Conj}[The]{Conjecture}
\newtheorem{Question}[The]{Question}
\newtheorem{Prob}[The]{Problem}
\newtheorem{Def}[The]{Definition}
\newtheorem{Not}[The]{Notation}
\newtheorem{Claim}[The]{Claim}
\newtheorem{Conc}[The]{Conclusion}
\newtheorem{Ex}[The]{Example}
\newtheorem{Fact}[The]{Fact}
\newtheorem{Formula}[The]{Formula}
\newtheorem{Formulae}[The]{Formulae}
\newtheorem{The-Def}[The]{Theorem and Definition}
\newtheorem{Prop-Def}[The]{Proposition and Definition}
\newtheorem{Cor-Def}[The]{Corollary and Definition}
\newtheorem{Conc-Def}[The]{Conclusion and Definition}
\newtheorem{Terminology}[The]{Note on terminology}
\newcommand{\C}{\mathbb{C}}
\newcommand{\R}{\mathbb{R}}
\newcommand{\N}{\mathbb{N}}
\newcommand{\Z}{\mathbb{Z}}
\newcommand{\Q}{\mathbb{Q}}
\newcommand{\Proj}{\mathbb{P}}
\newcommand{\Rc}{\mathcal{R}}
\newcommand{\Oc}{\mathcal{O}}
\newcommand{\Vc}{\mathcal{V}}
\newcommand{\Id}{\operatorname{Id}}
\newcommand{\pr}{\operatorname{pr}}
\newcommand{\rk}{\operatorname{rk}}
\newcommand{\del}{\partial}
\newcommand{\delbar}{\bar{\partial}}
\newcommand{\Cdot}{{\raisebox{-0.7ex}[0pt][0pt]{\scalebox{2.0}{$\cdot$}}}}
\newcommand\nilm{\Gamma\backslash G}
\newcommand\frg{{\mathfrak g}}
\newcommand{\fg}{\mathfrak g}
\newcommand{\Oh}{\mathcal{O}}
\newcommand{\Kur}{\operatorname{Kur}}
\newcommand\gc{\frg_\mathbb{C}}
\newcommand\jonas[1]{{\textcolor{green}{#1}}}
\newcommand\luis[1]{{\textcolor{red}{#1}}}
\newcommand\dan[1]{{\textcolor{blue}{#1}}}

\begin{center}

{\Large\bf Higher-Page Bott-Chern and Aeppli Cohomologies and Applications}

\end{center}

\begin{center}

{\large Dan Popovici, Jonas Stelzig and Luis Ugarte}

\end{center}

\vspace{1ex}

\noindent{\small{\bf Abstract.} For every positive integer $r$, we introduce two new cohomologies, that we call $E_r$-Bott-Chern and $E_r$-Aeppli, on compact complex manifolds. When $r=1$, they coincide with the usual Bott-Chern and Aeppli cohomologies, but they are coarser, respectively finer, than these when $r\geq 2$. They provide analogues in the Bott-Chern-Aeppli context of the $E_r$-cohomologies featuring in the Fr\"olicher spectral sequence of the manifold. We apply these new cohomologies in several ways to characterise the notion of page-$(r-1)$-$\partial\bar\partial$-manifolds that we introduced very recently. We also prove analogues of the Serre duality for these higher-page Bott-Chern and Aeppli cohomologies and for the spaces featuring in the Fr\"olicher spectral sequence. We obtain a further group of applications of our cohomologies to the study of Hermitian-symplectic and strongly Gauduchon metrics for which we show that they provide the natural cohomological framework.}

\vspace{2ex}

\section{Introduction}\label{section:Introduction}

Let $X$ be an $n$-dimensional compact complex manifold.

\vspace{1ex}

 (I)\, We recall that the {\bf Fr\"olicher spectral sequence (FSS)} of $X$ is a collection of complexes canonically associated with the complex structure of $X$. Its {\it $0$-th page} $$\dots\stackrel{\bar\partial}{\longrightarrow}C^\infty_{p,\,q}(X,\,\C)\stackrel{\bar\partial}{\longrightarrow}C^\infty_{p,\,q+1}(X,\,\C)\stackrel{\bar\partial}{\longrightarrow}\dots$$ is given by the Dolbeault complexes of $X$, consisting of the spaces of $\C$-valued pure-type $C^\infty$ differential forms on $X$ together with differentials $d_0$ defined by the $\bar\partial$-operator. Each of the next pages can be computed as the cohomology of the previous page. In particular, the {\it $1$-st page} $$\dots\stackrel{d_1}{\longrightarrow}E_1^{p,\,q}(X)\stackrel{d_1}{\longrightarrow}E_1^{p+1,\,q}(X)\stackrel{d_1}{\longrightarrow}\dots$$ features the Dolbeault cohomology groups $H^{p,\,q}_{\bar\partial}(X,\,\C)=E_1^{p,\,q}(X)$ of $X$ as vector spaces, while the differentials $d_1$ are induced by $\partial$ as $d_1(\{\alpha\}_{\bar\partial}) = \{\partial\alpha\}_{\bar\partial}$. In general, for every $r\in\N$, the differentials $d_r$ on the {\it $r$-th page} $$\dots\stackrel{d_r}{\longrightarrow}E_r^{p,\,q}(X)\stackrel{d_r}{\longrightarrow}E_r^{p+r,\,q-r+1}(X)\stackrel{d_r}{\longrightarrow}\dots$$ are of type $(r,\,-r+1)$.

Thus, the higher pages (namely, those corresponding to $r\geq 2$) of the FSS are refinements of the Dolbeault cohomology of $X$. The main interest of this spectral sequence is that it provides a link between the complex structure of $X$ and its differential structure. Indeed, the FSS converges to the {\it De Rham cohomology} of $X$ in the sense that there are (non-canonical) {\it isomorphisms}: $$H^k_{DR}(X,\,\C)\simeq\bigoplus\limits_{p+q=k}E_\infty^{p,\,q}(X), \hspace{3ex} k\in\{0,\dots , 2n\}.$$

Since $C^\infty_{p,\,q}(X,\,\C)=0$ unless $p,q\in \{0,...,n\}$, there exists a positive integer $r$ such that all the differentials $d_s$ vanish identically for all $s\geq r$. This is equivalent to saying that the spectral sequence becomes stationary from the $r$-th page in the sense that $$E_r^{p,\,q}(X) = E_{r+1}^{p,\,q}(X) = \dots , \hspace{3ex} p,q\in\{0,\dots , n\}.$$ This common vector space is denoted by $E_\infty^{p,\,q}(X)$. We say in this case that the spectral sequence {\it degenerates at the $r$-th page} (or {\it at $E_r$}) and we denote this property by $E_r(X) = E_\infty(X)$. Obviously, the smaller $r$, the stronger the degeneration property.

\vspace{2ex}

(II)\, On the other hand, the classical {\bf Bott-Chern cohomology} ([BC65]) provides a different kind of refinement of the Dolbeault cohomology, while the equally classical {\bf Aeppli cohomology} ([Aep62]) is a coarser replacement thereof. They have a wide range of applications, especially in non-K\"ahler geometry, but also in Arakelov geometry and in Index Theory. (See e.g. [Bis13] and the references therein.) They are respectively defined by $$H^{\bullet,\,\bullet}_{BC}(X,\,\C) = \frac{\ker\partial\cap\ker\bar\partial}{\mbox{Im}\,(\partial\bar\partial)} \hspace{3ex} \mbox{and} \hspace{3ex} H^{\bullet,\,\bullet}_A(X,\,\C) = \frac{\ker(\partial\bar\partial)}{\mbox{Im}\,\partial + \mbox{Im}\,\bar\partial}$$ and the identity induces natural maps
\[
\begin{tikzcd}
&H_{DR}^{p+q}(X,\,\C)\ar[rd]&\\
H_{BC}^{p,\, q}(X)\ar[ru]\ar[rd]&&H_A^{p,\, q}(X).\\
& E_1^{p,\, q}(X)\ar[ru]&
\end{tikzcd}
\]

Moreover, a Serre-type {\it duality} between the Bott-Chern cohomology of any bidegree and the {\it Aeppli cohomology} of the complementary bidegree (see e.g. [Sch07]), established by the {\it well-defined, canonical, non-degenerate} pairings: $$H^{p,\,q}_{BC}(X)\times H^{n-p,\,n-q}_A(X)\longrightarrow\C, \hspace{3ex} \bigg(\{\alpha\}_{BC},\,\{\beta\}_A\bigg)\longmapsto\int\limits_X\alpha\wedge\beta,$$ for all $p,q=0,\dots , n$, can be viewed as a generalisation to the transcendental context of the classical duality in algebraic geometry between {\it curves} and {\it divisors}. (Indeed, the current of integration on a divisor, respectively a curve, represents a cohomology class of bidegree $(1,\,1)$, respectively of bidegree $(n-1,\,n-1)$. The Serre duality between the Dolbeault cohomology groups of bidegrees $(1,\,1)$ and $(n-1,\,n-1)$ provides a framework for the divisor/curve duality.)

Thus, in many respects, the classical Bott-Chern and Aeppli cohomologies are related to the first page of the FSS, namely to the Dolbeault cohomology of $X$, to which they are canonically isomorphic on the important class of {\it $\partial\bar\partial$-manifolds}. This class was introduced by Deligne, Griffiths, Morgan and Sullivan [DGMS75], while this name was proposed in [Pop14, Definition 1.6].

\begin{Def}([DGMS75])\label{Def:dd-bar-lemma} A compact complex manifold $X$ is said to be a $\partial\bar\partial$-{\bf manifold} if for any $d$-closed {\it pure-type} form $u$ on $X$, the following exactness properties are equivalent: \\

\hspace{10ex} $u$ is $d$-exact $\Longleftrightarrow$ $u$ is $\partial$-exact $\Longleftrightarrow$ $u$ is $\bar\partial$-exact $\Longleftrightarrow$ $u$ is $\partial\bar\partial$-exact. 

\end{Def}

The $\partial\bar\partial$-manifolds behave cohomologically as compact K\"ahler manifolds, hence, in particular, they support a classical {\it Hodge decomposition}, namely {\it canonical isomorphisms}: $$H^k_{DR}(X,\,\C)\simeq\bigoplus\limits_{p+q=k}E_1^{p,\,q}(X), \hspace{3ex} k\in\{0,\dots , 2n\},$$ induced by the identity map using the fact that every Dolbeault cohomology class $\mathfrak{c}\in E_1^{\bullet,\,\bullet}(X)$ can be represented by a {\it $d$-closed} pure-type form on any $\partial\bar\partial$-manifold. In fact, the {\it $\partial\bar\partial$-manifolds} are characterised by the Hodge decomposition holding in the {\it canonical} sense just described.

But such an $X$ need not carry any K\"ahler metric and, in fact, the class of $\partial\bar\partial$-manifolds strictly contains the class of compact K\"ahler (and even that of {\it class ${\cal C}$}) manifolds.

\vspace{2ex}

(III)\, In [PSU20], we introduced, for every positive integer $r\in\N^\star$, the class of {\bf page-$(r-1)$-$\partial\bar\partial$-manifolds} by requiring the analogue of the Hodge decomposition to hold {\bf canonically} when the Dolbeault cohomology spaces $E_1^{\bullet,\,\bullet}(X)$ are replaced by the spaces $E_r^{\bullet,\,\bullet}(X)$ featuring on the $r$-th page of the FSS.

\begin{Def}([PSU20, Definition 2.7.])\label{Def:page_r-1_ddbar} Fix an arbitrary positive integer $r$. A compact $n$-dimensional complex manifold $X$ is said to be a {\bf page-$(r-1)$-$\partial\bar\partial$-manifold} if $X$ has the {\bf $E_r$-Hodge Decomposition} property in the sense that the identity map induces an isomorphism $$\bigoplus_{p+q=k}E_r^{p,\,q}(X)\simeq H^k_{DR}(X,\,\C)$$ for every $k\in\{0,\dots , 2n\}$.

  This means that the following two conditions are satisfied:

\vspace{1ex}

 (a)\, for every bidegree $(p,\,q)$ with $p+q=k$, every class $\{\alpha^{p,\,q}\}_{E_r}\in E_r^{p,\,q}(X)$ contains a {\bf $d$-closed representative of pure type} $\alpha^{p,\,q}\in C^\infty_{p,\,q}(X)$;

 (b)\, the linear map $$\bigoplus_{p+q=k}E_r^{p,\,q}(X)\ni\sum\limits_{p+q=k}\{\alpha^{p,\,q}\}_{E_r}\mapsto\bigg\{\sum\limits_{p+q=k}\alpha^{p,\,q}\bigg\}_{DR}\in H^k_{DR}(X,\,\C)$$

\noindent is {\bf well-defined} (in the sense that it does not depend on the choices of $d$-closed representatives $\alpha^{p,\,q}$ of the classes $\{\alpha^{p,\,q}\}_{E_r}$) and {\bf bijective}.

\end{Def}  

Among other things, it was proved in [PSU20, Theorem and Definition 1.2.] that the {\bf page-$(r-1)$-$\partial\bar\partial$-property} of $X$ is equivalent to the following two properties holding simultaneously: the Fr\"olicher spectral sequence of $X$ {\bf degenerates at $E_r$} and the De Rham cohomology of $X$ is {\bf pure}. (See Definition 2.4. and Note on terminology 2.5. in [PSU20] for the last notion.)

\vspace{1ex}

In particular, page-$0$-$\partial\bar\partial$-manifolds coincide with the classical $\partial\bar\partial$-manifolds, but for every $r\geq 1$, the class of page-$r$-$\partial\bar\partial$-manifolds is new. Moreover, the page-$r$-$\partial\bar\partial$-property becomes weaker and weaker as $r$ increases in the sense that, for every $r\in\N$, the following implication holds: \\

\hspace{6ex} $X$ is a page-$r$-$\partial\bar\partial$-manifold $\implies$ $X$ is a page-$(r+1)$-$\partial\bar\partial$-manifold.

\vspace{1ex}

Thus, the theory of page-$r$-$\partial\bar\partial$-manifolds enables the extension of Hodge Theory to the higher pages of the FSS and, in particular, a further extension of the class of compact K\"ahler manifolds. Quite a few examples of classes of page-$r$-$\partial\bar\partial$-manifolds that are not $\partial\bar\partial$-manifolds in the classical sense were exhibited in [PSU20].

\subsection{The new cohomologies}\label{subsection:introd_new_cohom} The main goal of this work is to continue the construction of the higher-page Hodge Theory begun in [PSU20] by introducing {\it higher-page analogues of the Bott-Chern and Aeppli cohomologies} that provide a natural link between their classical counterparts and the Fr\"olicher spectral sequence and, simultaneously, parallel the higher-page Fr\"olicher cohomologies $E_r^{\bullet,\,\bullet}(X)$ when $r\geq 2$.

The new terminology that we introduce in this paper can be divided into two main groups.

\vspace{1ex}

(a)\, A $(p,\,q)$-form $\alpha$ on $X$ is said to be {\it $E_r$-exact} if $\alpha$ represents the {\it zero $E_r$-cohomology class} on the $r$-th page of the Fr\"olicher spectral sequence of $X$. Also, $\alpha$ is said to be {\it $\overline{E}_r$-exact} if $\bar\alpha$ is {\it $E_r$-exact}. When $r=1$, these notions coincide with $\bar\partial$-exactness, respectively $\partial$-exactness.  

The notion of {\it $E_r\overline{E}_r$-exactness} is introduced in (iii) of Definition \ref{Def:E_rE_r-bar} as a weakening, inspired by the characterisations of $E_r$-exactness and $\overline{E}_r$-exactness given in Proposition \ref{Prop:E_r-closed-exact_conditions}, of the standard notion of $\partial\bar\partial$-exactness. When $r=1$, the notion of $E_1\overline{E}_1$-exactness is defined as $\partial\bar\partial$-exactness.

The notion of {\it $E_r\overline{E}_r$-closedness} is introduced in (i) of Definition \ref{Def:E_rE_r-bar} as a strengthening of the standard notion of $\partial\bar\partial$-closedness. When $r=1$, the notion of $E_1\overline{E}_1$-closedness is defined as $\partial\bar\partial$-closedness. 

All the exactness notions become weaker and weaker as $r$ increases, in contrast to the closedness conditions that become stronger and stronger. 

\vspace{2ex}

(b)\, The {\bf $E_r$-Bott-Chern}, respectively the {\bf $E_r$-Aeppli}, cohomology groups of bidegree $(p,\,q)$ of $X$, denoted by $E_{r,\, BC}^{p,\, q}(X)$ and $E_{r,\, A}^{p,\, q}(X)$, are introduced in Definition \ref{Def:E_r-BC_E_r-A} by quotienting out

\vspace{1ex}

-the smooth $d$-closed $(p,\,q)$-forms by the $E_r\overline{E}_r$-exact ones;

\vspace{1ex}

-respectively, the smooth $E_r\overline{E}_r$-closed $(p,\,q)$-forms by those lying in $\mbox{Im}\,\partial + \mbox{Im}\,\bar\partial$.

\vspace{1ex}

\noindent When $r=1$, these cohomology groups coincide with the standard Bott-Chern, respectively Aeppli, cohomology groups. Moreover, for every $(p,\,q)$, one has a sequence of canonical linear {\bf surjections}: \begin{equation}\label{eqn:BC_seq-surjections}H_{BC}^{p,\, q}(X)=E_{1,\, BC}^{p,\, q}(X)\twoheadrightarrow E_{2,\, BC}^{p,\, q}(X)\twoheadrightarrow\dots\twoheadrightarrow E_{r,\, BC}^{p,\, q}(X)\twoheadrightarrow E_{r+1,\, BC}^{p,\, q}(X)\twoheadrightarrow\dots\end{equation} and a  sequence of {\bf subspaces} of $H_A^{p,\, q}(X)$: \begin{equation}\label{eqn:A_seq-inclusions}\dots\subset E_{r+1,A}^{p,\,q}(X)\subset E_{r,\, A}^{p,\,q}(X)\subset\dots\subset E_{1,A}^{p,\,q}(X)=H_A^{p,\, q}(X).\end{equation}

  Moreover, it follows from Lemma \ref{Lem:E_rE_r-bar-properties} that the identity induces canonical linear maps:
\[
\begin{tikzcd}
&H_{DR}^{p+q}(X,\,\C)\ar[rd]&\\
E_{r,\, BC}^{p,\, q}(X)\ar[ru]\ar[rd]&&E_{r,\, A}^{p,\, q}(X)\\
& E_r^{p,\, q}(X)\ar[ru]&
\end{tikzcd}
\]

\noindent for every positive integer $r$ and every bidegree $(p,\,q)$.

\vspace{3ex}

Recall that $\partial\bar\partial$-manifolds are characterised by the equivalence of various types of exactness properties for $d$-closed pure-type forms (cf. Definition \ref{Def:dd-bar-lemma}). One of our main results in this paper, which also provides some of the main applications of the two kinds of cohomologies that we introduce, is the following statement, which extends this to page-$r$-$\partial\bar\partial$-manifolds (see Theorem \ref{The: page-r-Characterisations}).

\begin{The}\label{The:higher-page_BC-A-ddbar_introd} 
For a compact complex manifold $X$, denote by $E_{r,\, BC}^{p,\, q}(X)$ and $E_{r,\, A}^{p,\, q}(X)$ the higher-page Bott-Chern and Aeppli cohomology groups introduced in Definition \ref{Def:E_r-BC_E_r-A}. The following properties are equivalent:
\begin{enumerate}
	\item $X$ is a \textbf{page-$(r-1)$-$\partial\bar\partial$-manifold}.
	\item For any $d$-closed $(p,q)$-form $\alpha$, the following properties are equivalent:
	\[
	\alpha\text{ is } d\text{-exact}\Longleftrightarrow\alpha\text{ is } E_r\text{-exact}\Longleftrightarrow\alpha\text{ is } \bar E_r\text{-exact}\Longleftrightarrow\alpha\text{ is }E_r\bar E_r\text{-exact}.
	\]
	\item For all $p,q\in\{0,...,n\}$, the natural map $E_{r,\, BC}^{p,\, q}(X)\rightarrow E_{r,\, A}^{p,\, q}(X)$ is an isomorphism.
	\item For all $p,q\in\{0,...n\}$, the map $E_{r,\, BC}^{p,\, q}(X)\rightarrow E_{r,\, A}^{p,\, q}(X)$ is injective.
	\item One has $\dim E_{r,\, BC}^k(X)=\dim E_{r,\, A}^k(X)$ for all $k\in\{0,...,2n\}$.

	\end{enumerate}
\end{The}

\subsection{First set of further applications: duality results}\label{subsection:introd_duality}

Using the Hodge theory based on pseudo-differential Laplacians introduced in [Pop16] and [Pop19], we prove in $\S.$\ref{section:duality-Froelicher}, respectively $\S.$\ref{section:higher-page_BC-A}, the analogues of the Serre duality for the $E_r$-cohomology, respectively the $E_r$-Bott-Chern and $E_r$-Aeppli cohomologies, for all $r\geq 2$. The case $r=1$ is standard. These results can be summed up as follows (cf. Theorem \ref{The:duality_E2}, Corollaries \ref{Cor:well-definedness_Er_duality} and \ref{Cor:Er_duality}, Theorem \ref{The:E_r_BC-A_duality}).

\begin{The}\label{The:duality_introd} Let $X$ be a compact complex manifold with $\mbox{dim}_\C X=n$. Fix an arbitrary positive integer $r\in\N$. 

  For every $p,q\in\{0,\dots , n\}$, the canonical bilinear pairings

\vspace{1ex}

\hspace{12ex} $\displaystyle E_r^{p,\,q}(X)\times E_r^{n-p,\,n-q}(X)\longrightarrow\C, \hspace{3ex} (\{\alpha\}_{E_r},\,\{\beta\}_{E_r})\mapsto \int\limits_X\alpha\wedge\beta,$

\vspace{1ex}

\noindent and

\vspace{1ex}

\hspace{12ex} $\displaystyle E_{r,\,BC}^{p,\,q}(X)\times E_{r,\,A}^{n-p,\,n-q}(X)\longrightarrow\C, \hspace{3ex} (\{\alpha\}_{E_{r,\,BC}},\,\{\beta\}_{E_{r,\,A}})\mapsto \int\limits_X\alpha\wedge\beta,$

\noindent are {\bf well defined} and {\bf non-degenerate}.

\end{The}

This means that every space $E_r^{n-p,\,n-q}(X)$ can be viewed as the {\bf dual} of $E_r^{p,\,q}(X)$ and every space $E_{r,\,BC}^{p,\,q}(X)$ can be viewed as the {\bf dual} of $E_{r,\,A}^{n-p,\,n-q}(X)$.

Our method is analytical and Hodge-theoretical, although more algebraic approaches are, to a certain extend, possible (see Remark \ref{Rem:alternative-approaches_E_r-duality}). This  provides finer information useful for applications, e.g. in section \ref{section:applications_HS-SKT-sGG}.
\subsection{Second set of further applications: special metrics}\label{subsection:introd_special-metrics}

A final group of applications of the new cohomologies that we introduce in this paper features in section \ref{section:applications_HS-SKT-sGG} where it is shown that the higher-page Aeppli cohomologies provide the natural framework for the study of {\bf Hermitian-symplectic (H-S)} and {\bf strongly Gauduchon (sG)} metrics and manifolds. (See Proposition \ref{Prop:sG-HS_E_2A}.)

In particular, we introduce the {\it strongly Gauduchon (sG) cone} ${\cal SG}_X\subset E_{2,A}^{n-1,\,n-1}(X,\,\R)$ (in a way different to the one it was defined in [Pop15] and [PU18]), the {\it Hermitian-symplectic (H-S) cone} ${\cal HS}_X\subset E_{3,A}^{1,\,1}(X,\,\R)$ and the {\it SKT cone} ${\cal SKT}_X\subset E_{1,A}^{1,\,1}(X,\,\R) = H_A^{1,\,1}(X,\,\R)$ of $X$ (see Definition \ref{Def:E_rA_pos-cones}) that are then shown to be {\bf open convex cones} in their respective cohomology vector spaces (see Lemma \ref{Lem:E_rA_pos-cones}).

We go on to use these cones to obtain:

\vspace{1ex}

-a new numerical characterisation of {\it sGG manifolds} (introduced in [Pop15] and [PU18]) besides those obtained in [PU18] (one of which was $b_1=2h^{0,\,1}_{\bar\partial}$, while the inequality $b_1\leq 2h^{0,\,1}_{\bar\partial}$ holds on every manifold -- see Theorem 1.6 in [PU18]);

\vspace{1ex}

-a numerical characterisation of compact complex SKT manifolds on which every SKT metric is Hermitian-symplectic. 

\vspace{1ex}

Specifically, we prove the following fact (see Corollary \ref{Cor:cone-equalities} for further details), where $e$ stands each time for the dimension of the corresponding higher-page Aeppli cohomology space of $X$ denoted by $E$.

\begin{Prop}\label{Prop:introd_cone-equalities} Let $X$ be a compact complex manifold with $\mbox{dim}_\C X=n$. 

  \vspace{1ex}

  (i)\, The inequality $e_{2,A}^{n-1,\,n-1} \leq e_{1,A}^{n-1,\,n-1}$ holds. Moreover, $X$ is an {\bf sGG manifold} if and only if $e_{2,A}^{n-1,\,n-1} = e_{1,A}^{n-1,\,n-1}$.

  \vspace{1ex}

  (ii)\, The inequality $e_{3,A}^{1,\,1} \leq e_{1,A}^{1,\,1}$ holds. Moreover, {\bf every SKT metric on $X$ is Hermitian-symplectic (H-S)} if and only if $e_{3,A}^{1,\,1} = e_{1,A}^{1,\,1}$.

\end{Prop}

Recall that Streets and Tian asked in [ST10, Question 1.7] whether every H-S manifold of dimension $n\geq 3$ is K\"ahler. If the answer to this question is affirmative, then (ii) of Proposition \ref{Prop:introd_cone-equalities} implies that every SKT manifold $X$ for which $e_{3,A}^{1,\,1} = e_{1,A}^{1,\,1}$ is K\"ahler. 

\section{Serre-type duality for the Fr\"olicher spectral sequence}\label{section:duality-Froelicher}

Let $X$ be a compact complex manifold with $\mbox{dim}_\C X=n$. For every $r\in\N$, we let $E_r^{p,\,q}(X)$ stand for the space of bidegree $(p,\,q)$ featuring on the $r^{th}$ page of the Fr\"olicher spectral sequence of $X$. 

The first page of this spectral sequence is given by the Dolbeault cohomology of $X$, namely $E_1^{p,\,q}(X)=H^{p,\,q}_{\bar\partial}(X)$ for all $p,q$. Moreover, the classical Serre duality asserts that every space $H^{p,\,q}_{\bar\partial}(X)$ is the dual of $H^{n-p,\,n-q}_{\bar\partial}(X)$ via the canonical non-degenerate bilinear pairing $$H^{p,\,q}_{\bar\partial}(X)\times H^{n-p,\,n-q}_{\bar\partial}(X)\longrightarrow\C, \hspace{6ex} \bigg([\alpha]_{\bar\partial},\,[\beta]_{\bar\partial}\bigg)\mapsto \int\limits_X\alpha\wedge\beta.$$

In this section, we will extend this duality to all the pages of the Fr\"olicher spectral sequence. For the sake of perspicuity, we will first treat the case $r=2$ and then the more technically involved case $r\geq 3$.

\subsection{Serre-type duality for the second page of the Fr\"olicher spectral sequence}\label{subsection:duality_E2}

 The main ingredient in the proof of the next statement is the Hodge theory for the $E_2$-cohomology introduced in [Pop16] via the construction of a pseudo-differential Laplace-type operator $\widetilde\Delta$.

\begin{The}\label{The:duality_E2} For every $p,q\in\{0,\dots , n\}$, the canonical bilinear pairing \begin{equation}\label{eqn:duality_E2}E_2^{p,\,q}(X)\times E_2^{n-p,\,n-q}(X)\longrightarrow\C, \hspace{6ex} \bigg(\{\alpha\}_{E_2},\,\{\beta\}_{E_2}\bigg)\mapsto \int\limits_X\alpha\wedge\beta,\end{equation}

\noindent is {\bf well defined} (i.e. independent of the choices of representatives of the cohomology classes involved) and {\bf non-degenerate}. 

\end{The}

\noindent {\it Proof.} $\bullet$ To prove {\it well-definedness}, let $\{\alpha\}_{E_2}\in E_2^{p,\,q}(X)$ and $\{\beta\}_{E_2}\in E_2^{n-p,\,n-q}(X)$ be arbitrary classes in which we choose arbitrary representatives $\alpha, \beta$. Thus, $\bar\partial\alpha=0$, $\partial\alpha\in\mbox{Im}\,\bar\partial$ (since $[\alpha]_{\bar\partial}\in\ker d_1$) and $\beta$ has the analogous properties. In particular, $\partial\beta = \bar\partial v$ for some $(n-p+1,\,n-q-1)$-form $v$. Any other representative of the class $\{\alpha\}_{E_2}$ is of the shape $\alpha + \partial\eta + \bar\partial\zeta$ for some $(p-1,\,q)$-form $\eta\in\ker\bar\partial$ and some $(p,\,q-1)$-form $\zeta$. (Indeed, $[\partial\eta]_{\bar\partial} = d_1([\eta]_{\bar\partial})$.) We have

\begin{eqnarray*}\label{eqn:well-definedness_check} \int\limits_X(\alpha + \partial\eta + \bar\partial\zeta)\wedge\beta & = & \int\limits_X\alpha\wedge\beta + (-1)^{p+q}\,\int\limits_X\eta\wedge\partial\beta + (-1)^{p+q}\,\int\limits_X\zeta\wedge\bar\partial\beta \hspace{6ex} (\mbox{by Stokes})  \\
 & = & \int\limits_X\alpha\wedge\beta + (-1)^{p+q}\,\int\limits_X\eta\wedge\bar\partial v   \hspace{6ex} (\mbox{since}\hspace{1ex} \partial\beta = \bar\partial v \hspace{2ex}\mbox{and}\hspace{1ex} \bar\partial\beta=0) \\
 & = & \int\limits_X\alpha\wedge\beta + \int\limits_X\bar\partial\eta\wedge v =\int\limits_X\alpha\wedge\beta    \hspace{10ex} (\mbox{by Stokes and} \hspace{1ex} \bar\partial\eta=0).\end{eqnarray*}

\noindent Similarly, the integral $\int_X\alpha\wedge\beta$ does not change if $\beta$ is replaced by $\beta + \partial a + \bar\partial b$ with $a\in\ker\bar\partial$.  

\vspace{1ex}

 $\bullet$ To prove {\it non-degeneracy} for the pairing (\ref{eqn:duality_E2}), we fix an arbitrary Hermitian metric $\omega$ on $X$ and use the pseudo-differential Laplacian associated with $\omega$ introduced in [Pop16, $\S.1$]:

$$\widetilde{\Delta} : = \partial p''\partial^{\star} + \partial^{\star}p''\partial + \bar\partial\bar\partial^{\star} + \bar\partial^{\star} \bar\partial:C^{\infty}_{p,\,q}(X)\longrightarrow C^{\infty}_{p,\,q}(X), \hspace{2ex} p,q=0,\dots , n,$$

\noindent where $p'': C^\infty_{p,\,q}(X)\longrightarrow{\cal H}^{p,\,q}_{\Delta''}(X):=\ker\Delta''$ is the orthogonal projection w.r.t. the $L^2$ inner product defined by $\omega$ onto the $\Delta''$-harmonic space in the standard $3$-space decomposition 

$$C^\infty_{p,\,q}(X) = {\cal H}^{p,\,q}_{\Delta''}(X) \oplus \mbox{Im}\,\bar\partial \oplus \mbox{Im}\,\bar\partial^\star.$$

\noindent Recall that $\Delta''= \bar\partial\bar\partial^{\star} +\bar\partial^{\star}\bar\partial:C^{\infty}_{p,\,q}(X)\longrightarrow C^{\infty}_{p,\,q}(X)$ is the usual $\bar\partial$-Laplacian induced by $\omega$ and the above decomposition is $L^2_\omega$-orthogonal. It was proved in [Pop16, Theorem 1.1.] that for every $p,q\in\{0,\dots , n\}$, the linear map

$${\cal H}^{p,\,q}_{\widetilde\Delta}(X):=\ker(\widetilde\Delta:C^\infty_{p,\,q}(X)\longrightarrow C^\infty_{p,\,q}(X))\longrightarrow E_2^{p,\,q}(X), \hspace{3ex} \alpha\mapsto\{\alpha\}_{E_2},$$

\noindent is an {\bf isomorphism}. This is a Hodge isomorphism showing that every double class $\{\alpha\}_{E_2}\in E_2^{p,\,q}(X)$ contains a unique $\widetilde\Delta$-harmonic representative.

\begin{Claim}\label{Claim:Hodge-star_harmonicity} For every $\alpha\in C^{\infty}_{p,\,q}(X)$, the equivalence holds: $\widetilde\Delta\alpha = 0 \iff \widetilde\Delta(\star\bar\alpha)=0,$ where $\star=\star_\omega$ is the Hodge-star operator associated with $\omega$. 

\end{Claim}

Suppose for a moment that this claim has been proved. To prove non-degeneracy for the pairing (\ref{eqn:duality_E2}), let $\{\alpha\}_{E_2}\in E_2^{p,\,q}(X)$ be an arbitrary non-zero class whose unique $\widetilde\Delta$-harmonic representative is denoted by $\alpha$. So, $\alpha\neq 0$ and $\star\bar\alpha\in{\cal H}^{n-p,\,n-q}_{\widetilde\Delta}(X)\setminus\{0\}$. In particular, $\star\bar\alpha$ represents an element in $E_2^{n-p,\,n-q}(X)$ and the pair $(\{\alpha\}_{E_2},\,\{\star\bar\alpha\}_{E_2})$ maps under (\ref{eqn:duality_E2}) to $\int_X\alpha\wedge\star\bar\alpha = \int_X|\alpha|^2_\omega\,dV_\omega = ||\alpha||^2_{L^2_\omega}\neq 0$. Since $p,q$ and $\alpha$ were arbitrary, we conclude that the pairing (\ref{eqn:duality_E2}) is non-degenerate.

\vspace{1ex}

$\bullet$ {\it Proof of Claim \ref{Claim:Hodge-star_harmonicity}.} Since $\widetilde\Delta$ is a sum of non-negative operators of the shape $A^\star A$, we have $$\ker\widetilde\Delta = \ker(p''\partial)\cap\ker(p''\partial^\star)\cap\ker\bar\partial\cap\ker\bar\partial^\star.$$

\noindent Thus, the orthogonal $3$-space decomposition recalled above yields the following equivalence: $$\alpha\in\ker\widetilde\Delta \iff (i)\hspace{1ex}\partial\alpha\in\mbox{Im}\,\bar\partial \oplus \mbox{Im}\,\bar\partial^\star, \hspace{1ex} (ii)\hspace{1ex} \partial^\star\alpha\in\mbox{Im}\,\bar\partial \oplus \mbox{Im}\,\bar\partial^\star  \hspace{2ex} \mbox{and} \hspace{2ex} (iii)\hspace{1ex} \alpha\in\ker\bar\partial\cap\ker\bar\partial^\star.$$

 Let $\alpha\in\ker\widetilde\Delta$. Since $\star:\Lambda^{p,\,q}T^\star X\longrightarrow \Lambda^{n-q,\,n-p}T^\star X$ is an isomorphism, the well-known identities $\star\star = (-1)^{p+q}$ on $(p,\,q)$-forms, $\partial^\star = -\star\bar\partial\star$ and $\bar\partial^\star = -\star\partial\star$ yield:

\vspace{1ex}

$\bar\partial\alpha = 0 \iff \partial\bar\alpha = 0 \iff \bar\partial^\star(\star\bar\alpha) = 0  \hspace{2ex} \mbox{and} \hspace{2ex} \bar\partial^\star\alpha = 0 \iff \partial^\star\bar\alpha = 0 \iff \bar\partial(\star\bar\alpha) = 0.$

\vspace{1ex}

\noindent Thus, $\alpha$ satisfies condition $(iii)$ if and only if $\star\bar\alpha$ satisfies condition $(iii)$.

 Meanwhile, $\alpha$ satisfies condition $(ii)$ if and only if there exist forms $\xi,\eta$ such that $\partial^\star\alpha = \bar\partial\xi + \bar\partial^\star\eta$. The last identity is equivalent to

\vspace{1ex}

$\bar\partial^\star\bar\alpha = \partial\bar\xi + \partial^\star\bar\eta \iff -(\star\star)\partial(\star\bar\alpha) = \pm\,\star\partial\star(\star\bar\xi) \pm\, \star(-\star\bar\partial\star\bar\eta) \iff \partial(\star\bar\alpha) = \pm\,\bar\partial^\star(\star\bar\xi) \pm \bar\partial(\star\bar\eta).$

\vspace{1ex} 

\noindent Thus, $\alpha$ satisfies condition $(ii)$ if and only if $\star\bar\alpha$ satisfies condition $(i)$. 

 Similarly, $\alpha$ satisfies condition $(i)$ if and only if there exist forms $u, v$ such that $\partial\alpha = \bar\partial u + \bar\partial^\star v$. The last identity is equivalent to

\vspace{1ex}

$\bar\partial\bar\alpha = \partial\bar{u} + \partial^\star\bar{v} \iff -\star\bar\partial\star(\star\bar\alpha) = -\star\partial(\star\star\bar{u}) - \star\partial^\star(\star\star\bar{v}) \iff \partial^\star(\star\bar\alpha) = \bar\partial^\star(\star\bar{u}) + \bar\partial(\star\bar{v}).$

\noindent Thus, $\alpha$ satisfies condition $(i)$ if and only if $\star\bar\alpha$ satisfies condition $(ii)$.

 This completes the proof of Claim \ref{Claim:Hodge-star_harmonicity} and implicitly that of Theorem \ref{The:duality_E2}.  \hfill $\Box$.

\subsection{Serre-type duality for the pages $r\geq 3$ of the Fr\"olicher spectral sequence}\label{subsection:duality_Er}

In this section, we construct elliptic pseudo-differential operators $\widetilde\Delta_{(r)}^{(\omega)}$ associated with any given Hermitian metric $\omega$ on $X$ whose kernels are isomorphic to the spaces $E_r^{p,\,q}(X)$ in every bidegree $(p,\,q)$. This extends to arbitrary positive integers $r\in\N_{>0}$ the construction performed in [Pop16] for $r=2$. We then apply this construction to prove the existence of a (non-degenerate) duality between every space $E_r^{p,\,q}(X)$ and the space $E_r^{n-p,\,n-q}(X)$ that extends to every page in the Fr\"olicher spectral sequence the classical Serre duality (corresponding to $r=1$).

Let $X$ be an arbitrary compact complex $n$-dimensional manifold. Fix $r\in\N$ and a bidegree $(p,\,q)$ with $p,q\in\{0,\dots, n\}$. A smooth $\C$-valued $(p,\,q)$-form $\alpha$ on $X$ will be said to be {\it $E_r$-closed} if it represents an $E_r$-cohomology class, denoted by $\{\alpha\}_{E_r}\in E_r^{p,\,q}(X)$, on the $r^{th}$ page of the Fr\"olicher spectral sequence of $X$. Meanwhile, $\alpha$ will be said to be {\it $E_r$-exact} if it represents the {\it zero} $E_r$-cohomology class, i.e. if $\{\alpha\}_{E_r}=0\in E_r^{p,\,q}(X)$. The $\C$-vector space of $C^\infty$ $E_r$-closed (resp. $E_r$-exact) $(p,\,q)$-forms will be denoted by ${\cal Z}_r^{p,\,q}(X)$ (resp. ${\mathscr C}_r^{p,\,q}(X)$). Of course, ${\mathscr C}_r^{p,\,q}(X)\subset{\cal Z}_r^{p,\,q}(X)$ and $E_r^{p,\,q}(X) = {\cal Z}_r^{p,\,q}(X)/{\mathscr C}_r^{p,\,q}(X)$.

The following statement was proved in [CFGU97]. It renders explicit the $E_r$-closedness and $E_r$-exactness conditions. In particular, it gives a more concrete description, equivalent to the more formal standard one, of the spaces $E_r^{p,\,q}(X)$ and the differentials $d_r$ featuring in the Fr\"olicher spectral sequence. It was also used in [Pop19].

\begin{Prop}\label{Prop:E_r-closed-exact_conditions} (i)\, Fix $r\geq 2$. A form $\alpha\in C^\infty_{p,\,q}(X)$ is {\bf $E_r$-closed} if and only if there exist forms $u_l\in C^\infty_{p+l,\,q-l}(X)$ with $l\in\{1,\dots , r-1\}$ satisfying the following tower of $r$ equations: \begin{eqnarray*}\label{eqn:tower_E_r-closedness} \bar\partial\alpha & = & 0 \\
     \partial\alpha & = & \bar\partial u_1 \\
     \partial u_1 & = & \bar\partial u_2 \\
     \vdots & & \\
     \partial u_{r-2} & = & \bar\partial u_{r-1}.\end{eqnarray*}

We say in this case that $\bar\partial\alpha=0$ and $\partial\alpha$ {\bf runs at least $(r-1)$ times}.

  (ii)\, Fix $r\geq 2$. The map $d_r: E_r^{p,\,q}(X)\longrightarrow E_r^{p+r,\,q-r+1}(X)$ acts as $d_r(\{\alpha\}_{E_r}) = \{\partial u_{r-1}\}_{E_r}$ for every $E_r$-class $\{\alpha\}_{E_r}$, any representative $\alpha$ thereof and any choice of forms $u_l$ satisfying the above tower of $E_r$-closedness equations for $\alpha$.

\vspace{1ex}

  (iii)\, Fix $r\geq 2$. A form $\alpha\in C^\infty_{p,\,q}(X)$ is {\bf $E_r$-exact} if and only if there exist forms $\zeta\in C^\infty_{p-1,\,q}(X)$ and $\xi\in C^\infty_{p,\,q-1}(X)$ such that $$\alpha=\partial\zeta + \bar\partial\xi,$$

\noindent with $\xi$ arbitrary and $\zeta$ satisfying the following tower of $(r-1)$ equations: \begin{eqnarray}\label{eqn:tower_E_r-exactness_l}\nonumber \bar\partial\zeta & = & \partial v_{r-3}  \\
    \nonumber \bar\partial v_{r-3} & = & \partial v_{r-4} \\
    \nonumber \vdots & & \\
    \nonumber \bar\partial v_1 & = & \partial v_0 \\
    \nonumber           \bar\partial v_0 & = & 0,\end{eqnarray}

\noindent for some forms $v_0,\dots , v_{r-3}$. (When $r=2$, $\zeta_{r-2} = \zeta_0$ must be $\bar\partial$-closed.)

We say in this case that $\bar\partial\zeta$ {\bf reaches $0$ in at most $(r-1)$ steps}.

\vspace{1ex}

(iv)\, The following inclusions hold in every bidegree $(p,\,q)$:

$$\dots\subset{\mathscr C}_r^{p,\,q}(X)\subset{\mathscr C}_{r+1}^{p,\,q}(X)\subset\dots\subset{\cal Z}_{r+1}^{p,\,q}(X)\subset{\cal Z}_r^{p,\,q}(X)\subset\dots,$$

\noindent with $\{0\}={\mathscr C}_0^{p,\,q}(X)\subset{\mathscr C}_1^{p,\,q}(X)=(\mbox{Im}\,\bar\partial)^{p,\,q}$ and ${\cal Z}_1^{p,\,q}(X)=(\ker\bar\partial)^{p,\,q}\subset{\cal Z}_0^{p,\,q}(X)=C^\infty_{p,\,q}(X)$.

\end{Prop}

\noindent {\it Proof.} See [CFGU97].   \hfill $\Box$

\vspace{2ex}

The immediate consequence that we notice is the well-definedness of the pairing that parallels on any page of the Fr\"olicher spectral sequence the classical Serre duality.

\begin{Cor}\label{Cor:well-definedness_Er_duality} Let $X$ be a compact complex manifold with $\mbox{dim}_\C X=n$. For every positive integer $r\in\N_{>0}$ and every $p,q\in\{0,\dots , n\}$, the canonical bilinear pairing  \begin{equation*} E_r^{p,\,q}(X)\times E_r^{n-p,\,n-q}(X)\longrightarrow\C, \hspace{3ex} (\{\alpha\}_{E_r},\,\{\beta\}_{E_r})\mapsto \int\limits_X\alpha\wedge\beta,\end{equation*} \noindent is {\bf well defined} (i.e. independent of the choices of representatives of the $E_r$-classes involved).

\end{Cor}

\noindent {\it Proof.} By symmetry, it suffices to prove that $\int_X\alpha\wedge\beta=0$ whenever $\alpha\in C^\infty_{p,\,q}(X)$ is $E_r$-exact and $\beta\in C^\infty_{n-p,\,n-q}(X)$ is $E_r$-closed. By Proposition \ref{Prop:E_r-closed-exact_conditions}, these conditions are equivalent to $$\bar\partial\beta = 0, \hspace{2ex} \partial\beta = \bar\partial u_1, \hspace{2ex}\partial u_1 = \bar\partial u_2, \dots\hspace{2ex} ,\partial u_{r-2} = \bar\partial u_{r-1},$$

\noindent for some forms $u_j$ and to $\alpha=\partial\zeta + \bar\partial\xi$ for some form $\zeta$ satisfying $$\bar\partial\zeta = \partial v_{r-3}, \hspace{2ex} \bar\partial v_{r-3} = \partial v_{r-4},\dots \hspace{2ex}, \bar\partial v_1 = \partial v_0,\bar\partial v_0 =0$$

\noindent for some forms $v_k$. We get $$\int\limits_X\alpha\wedge\beta = \int\limits_X\partial\zeta\wedge\beta + \int\limits_X\bar\partial\xi\wedge\beta.$$

Every integral on the r.h.s. above is seen to vanish by repeated integration by parts. Specifically, $\int_X\bar\partial\xi\wedge\beta = \pm \int\limits_X\xi\wedge\bar\partial\beta =0$ since $\bar\partial\beta = 0$, while for every $l\in\{1,\dots , r-2\}$ we have

\begin{eqnarray*}\int\limits_X\partial\zeta\wedge\beta & = & \pm\int\limits_X\zeta\wedge\partial\beta = \pm\int\limits_X\zeta\wedge\bar\partial u_1 = \pm\int\limits_X\bar\partial\zeta\wedge u_1 = \pm\int\limits_X\partial v_{r-3}\wedge u_1 \\
   & = & \pm\int\limits_X v_{r-3}\wedge \partial u_1 = \pm\int\limits_X v_{r-3}\wedge\bar\partial u_2 = \pm\int\limits_X \bar\partial v_{r-3}\wedge u_2 = \pm\int\limits_X \partial v_{r-4}\wedge u_2  \\
    & \vdots &  \\
  & = & \pm\int\limits_X v_0\wedge \partial u_{r-2} = \pm\int\limits_X v_0\wedge \bar\partial u_{r-1} = \pm\int\limits_X \bar\partial v_0\wedge u_{r-1} =0,\end{eqnarray*}

\noindent since $\bar\partial v_0=0$.  \hfill $\Box$

\vspace{3ex}

We will now prove that the above pairing is also non-degenerate, thus defining a Serre-type duality on every page of the Fr\"olicher spectral sequence. Much of the following preliminary discussion appeared in [Pop19, $\S.2.2$ and Appendix], so we will only recall the bare bones.

Let us fix an arbitrary Hermitian metric $\omega$ on $X$. For every bidegree $(p,\,q)$, {\it $\omega$-harmonic spaces} (also called {\it$E_r$-harmonic spaces}): $$\cdots\subset{\cal H}_{r+1}^{p,\,q}\subset{\cal H}_r^{p,\,q}\subset\cdots\subset{\cal H}_1^{p,\,q}\subset C^\infty_{p,\,q}(X)$$

\noindent were inductively constructed in [Pop17, $\S.3.2$, especially Definition 3.3. and Corollary 3.4.] such that every subspace ${\cal H}_r^{p,\,q}={\cal H}_r^{p,\,q}(X,\,\omega)$ is isomorphic to the corresponding space $E_r^{p,\,q}(X)$ on the $r^{th}$ page of the Fr\"olicher spectral sequence.

Moreover, these spaces fit into the inductive construction described in the next

\begin{Prop}\label{Prop:H_r_construction} Let $(X,\,\omega)$ be a compact Hermitian manifold with $\mbox{dim}_\C X=n$.

  \vspace{1ex}

  (i)\, For every bidegree $(p,\,q)$, the space $C^\infty_{p,\,q}(X)$ splits successively into mutually $L^2_\omega$-orthogonal subspaces as follows:

  \begin{eqnarray*} C^\infty_{p,\,q}(X) = \mbox{Im}\,d_0\hspace{1ex} \oplus & \underbrace{{\cal H}_1^{p,\,q}}_{{ }\rotatebox{90}{=}} & \oplus \hspace{1ex} \mbox{Im}\,d_0^\star \\
     & \overbrace{\mbox{Im}\,d_1^{(\omega)}\hspace{1ex}\oplus\hspace{1ex} \underbrace{{\cal H}_2^{p,\,q}}_{{ }\rotatebox{90}{=}} \hspace{1ex}\oplus\hspace{1ex}\mbox{Im}\,(d_1^{(\omega)})^\star} & \\
     & \vdots  &  \\
      & \rotatebox{90}{=}  & \\
     & \overbrace{\mbox{Im}\,d_{r-1}^{(\omega)}\hspace{1ex}\oplus\hspace{1ex} \underbrace{{\cal H}_r^{p,\,q}}_{{ }\rotatebox{90}{=}} \hspace{1ex}\oplus\hspace{1ex}\mbox{Im}\,(d_{r-1}^{(\omega)})^\star} &  \\
     & \overbrace{\mbox{Im}\,d_r^{(\omega)}\hspace{1ex}\oplus\hspace{1ex} \underbrace{{\cal H}_{r+1}^{p,\,q}}_{{ }\rotatebox{90}{=}} \hspace{1ex}\oplus\hspace{1ex}\mbox{Im}\,(d_r^{(\omega)})^\star} &  \\
   & \vdots  &   \end{eqnarray*}

  \noindent where, for $r\in\N_{>0}$, the operators $d_r^{(\omega)}$ are defined as \begin{equation}\label{eqn:d_r_metric-realisation-def}d_r^{(\omega)} = d_r^{(\omega){p,\,q}} =p_r\partial D_{r-1}p_r  : {\cal H}_r^{p,\,q} \longrightarrow {\cal H}_r^{p+r,\,q-r+1}\end{equation}

\noindent using the $L^2_\omega$-orthogonal projections $p_r = p_r^{p,\,q}:C^\infty_{p,\,q}(X)\longrightarrow{\cal H}_r^{p,\,q}$ onto the $\omega$-harmonic spaces ${\cal H}_r^{p,\,q}$ and where we inductively define $$D_{r-1}:=((\widetilde\Delta^{(1)})^{-1}\bar\partial^\star\partial)\dots((\widetilde\Delta^{(r-1)})^{-1}\bar\partial^\star\partial) \hspace{3ex} \mbox{and} \hspace{3ex} D_0=\mbox{Id}.$$ (So, $p_1=p''$.) See (iii) below for the inductive definition of the pseudo-differential Laplacians $\widetilde\Delta^{(r)}$.

 Thus, the triples $(p_r,\,d_r^{(\omega)},\, {\cal H}_{r+1}^{p,\,q})$ are defined by induction on $r\in\N_{>0}$: once the triple $(p_{r-1},\,d_{r-1}^{(\omega)},\, {\cal H}_r^{p,\,q})$ has been constructed for all the bidegrees $(p,\,q)$, it induces $p_r$, which induces $d_r^{(\omega)}$, which induces ${\cal H}_{r+1}^{p,\,q}$ defined as the $L^2_\omega$-orthogonal complement of $\mbox{Im}\,d_r^{(\omega)}$ in $\ker\,d_r^{(\omega)}$.

The operators $d_r^{(\omega)}$ can also be considered to be defined on the whole spaces of smooth forms: $$d_r^{(\omega)}=p_r\partial D_{r-1}p_r : C^\infty_{p,\,q}(X) \longrightarrow C^\infty_{p+r,\,q-r+1}(X).$$

  \vspace{1ex}

  (ii)\, The above definition of $d_r^{(\omega)}$ follows from the requirement that the following diagram be commutative:

\vspace{3ex}

\hspace{30ex} $\begin{CD}
E_r^{p,\,q}(X)              @>d_r>>          E_r^{p+r,\,q-r+1}(X) \\
@V\simeq VV                                       @V\simeq VV \\
{\cal H}_r^{p,\,q}          @>d_r^{(\omega)} = p_r\partial D_{r-1}p_r>>      {\cal H}_r^{p+r,\,q-r+1},\end{CD}$

\vspace{3ex}

\noindent where the maps $d_r:E_r^{p,\,q}(X)\longrightarrow E_r^{p+r,\,q-r+1}(X)$ are the differentials on the $r^{th}$ page of the Fr\"olicher spectral sequence. Thus, the maps $d_r^{(\omega)}$ are the metric realisations, at the level of the harmonic spaces, of the canonical maps $d_r$.

\vspace{2ex}

  (iii)\, For every $r\in\N_{>0}$, the adjoint of $d_r^{(\omega)}$ is

$$(d_r^{(\omega)})^\star = p_rD_{r-1}^\star\partial^\star p_r:{\cal H}_r^{p+r,\,q-r+1}\longrightarrow{\cal H}_r^{p,\,q}.$$

\noindent It induces the ``Laplacian'' $$\widetilde\Delta_{(r+1)}^{(\omega)}= d_r^{(\omega)}\,(d_r^{(\omega)})^\star + (d_r^{(\omega)})^\star\,d_r^{(\omega)}:{\cal H}_r^{p,\,q}\longrightarrow{\cal H}_r^{p,\,q}$$

\noindent given by the explicit formula $$\widetilde\Delta_{(r+1)}^{(\omega)} = p_r\,[(\partial D_{r-1}p_r)\,(\partial D_{r-1}p_r)^\star + (p_r\partial D_{r-1})^\star\,(p_r\partial D_{r-1}) + \widetilde\Delta^{(r)}]\,p_r,$$

\noindent which is the restriction and co-restriction to ${\cal H}_r^{p,\,q}$ of the pseudo-differential Laplacian \begin{eqnarray*} \widetilde\Delta^{(r+1)} & : = & (\partial D_{r-1}p_r)\,(\partial D_{r-1}p_r)^\star + (p_r\partial D_{r-1})^\star\,(p_r\partial D_{r-1}) + \widetilde\Delta^{(r)} : C^{\infty}_{p,\,q}(X)\longrightarrow C^{\infty}_{p,\,q}(X).\end{eqnarray*}

\vspace{2ex}

(iv)\, For every $r\in\N_{>0}$ and every bidegree $(p,\,q)$, the following orthogonal $3$-space decomposition holds:

$${\cal H}_r^{p,\,q} = \mbox{Im}\,d_r^{(\omega)}\oplus{\cal H}_{r+1}^{p,\,q}\oplus\mbox{Im}\,(d_r^{(\omega)})^\star,$$

\noindent where $\ker d_r^{(\omega)}=\mbox{Im}\,d_r^{(\omega)}\oplus{\cal H}_{r+1}^{p,\,q}$. In particular, this confirms that ${\cal H}_{r+1}^{p,\,q}$ is the orthogonal complement for the $L^2_\omega$-inner product of $\mbox{Im}\,d_r^{(\omega)}$ in $\ker d_r^{(\omega)}$. Moreover,

$${\cal H}_{r+1}^{p,\,q}=\ker\widetilde\Delta_{(r+1)}^{(\omega)}=\ker d_r^{(\omega)}\cap\ker (d_r^{(\omega)})^\star\simeq E_{r+1}^{p,\,q}(X),$$

\noindent for every $r\in\N$ and all $p,q\in\{0,\dots , n\}$.

\end{Prop}

\noindent {\it Proof.} The verification of the details of these statements was done in [Pop19, $\S.2.2$ and Appendix].   \hfill $\Box$

\vspace{3ex}

We saw in (i) of Proposition \ref{Prop:E_r-closed-exact_conditions} how the $E_r$-closedness property of a differential form is characterised in explicit terms. We will now define by analogy the property of $E_r^\star$-closedness when a Hermitian metric has been fixed.

\begin{Def}\label{Def:E_r-closed-star} Let $(X,\,\omega)$ be an $n$-dimensional compact complex Hermitian manifold. Fix $r\geq 1$ and a bidegree $(p,\,q)$. A form $\alpha\in C^\infty_{p,\,q}(X)$ is said to be {\bf $E_r^\star$-closed} with respect to the metric $\omega$ if and only if there exist forms $v_l\in C^\infty_{p-l,\,q+l}(X)$ with $l\in\{1,\dots , r-1\}$ satisfying the following tower of $r$ equations: \begin{eqnarray*}\label{eqn:tower_E_r-closedness-star} \bar\partial^\star\alpha & = & 0 \\
     \partial^\star\alpha & = & \bar\partial^\star v_1 \\
     \partial^\star v_1 & = & \bar\partial^\star v_2 \\
     \vdots & & \\
     \partial^\star v_{r-2} & = & \bar\partial^\star v_{r-1}.\end{eqnarray*}

We say in this case that $\bar\partial^\star\alpha=0$ and $\partial^\star\alpha$ {\bf runs at least $(r-1)$ times}.

\end{Def}

We can now use the $E_r$-closedness and $E_r^\star$-closedness properties to characterise the ${\cal H}_r$-harmonicity property defined above.

\begin{Prop}\label{Prop:E_r-closedness-star-harmonicity} Let $(X,\,\omega)$ be an $n$-dimensional compact complex Hermitian manifold. Fix $r\geq 1$ and a bidegree $(p,\,q)$. For any form $\alpha\in C^\infty_{p,\,q}(X)$, the following equivalence holds: $$\alpha\in{\cal H}_r^{p,\,q} \iff \alpha \hspace{1ex}\mbox{is}\hspace{1ex} E_r\mbox{-closed} \hspace{1ex}\mbox{and}\hspace{1ex} E_r^\star\mbox{-closed}.$$

\end{Prop}

\noindent {\it Proof.} We know from Proposition \ref{Prop:H_r_construction} that $\alpha\in{\cal H}_{r+1}^{p,\,q}$ if and only if $\alpha\in{\cal H}_r^{p,\,q}$ and $\alpha\in\ker d_r^{(\omega)}\cap\ker(d_r^{(\omega)})^\star$. Now, for $\alpha\in{\cal H}_r^{p,\,q}$, the definition of $d_r^{(\omega)}$ shows that $\alpha\in\ker d_r^{(\omega)}$ if and only if $\alpha\in\ker(p_r\partial D_{r-1})$ and this last fact is equivalent to $\alpha$ being $E_{r+1}$-closed. Similarly, for $\alpha\in{\cal H}_r^{p,\,q}$, the definition of $(d_r^{(\omega)})^\star$ shows that $\alpha\in\ker(d_r^{(\omega)})^\star$ if and only if $\alpha\in\ker(\partial D_{r-1}p_r)^\star$ and this last fact is equivalent to $\alpha$ being $E_{r+1}^\star$-closed.   \hfill $\Box$

\vspace{3ex}

\begin{Cor}\label{Cor:E_r_duality} In the setting of Proposition \ref{Prop:E_r-closedness-star-harmonicity}, the following equivalence holds: $$\alpha \hspace{1ex}\mbox{is}\hspace{1ex} E_r\mbox{-closed} \iff \star\bar\alpha \hspace{1ex}\mbox{is}\hspace{1ex} E_r^\star\mbox{-closed}.$$

\end{Cor}

\noindent {\it Proof.} We know from (i) of Proposition \ref{Prop:E_r-closed-exact_conditions} that $\alpha$ is $E_r$-closed if and only if there exist forms $u_l\in C^\infty_{p+l,\,q-l}(X)$ for $l=1,\dots , r-1$ such that $$(-\star\partial\star)\,\star\bar\alpha=0, \hspace{1ex} (-\star\bar\partial\star)\,\star\bar\alpha=(-\star\partial\star)\,\star\bar{u}_1, \dots , (-\star\bar\partial\star)\,\star\bar{u}_{r-2}=(-\star\partial\star)\,\star\bar{u}_{r-1}.$$

\noindent Indeed, we have transformed the $E_r$-closedness condition of (i) in Proposition \ref{Prop:E_r-closed-exact_conditions} by conjugating and applying the Hodge star operator several times. Since $-\star\partial\star = \bar\partial^\star$ and $-\star\bar\partial\star = \partial^\star$, the above conditions are equivalent to $\star\bar\alpha$ being $E_r^\star$-closed (with the forms $\star\bar{u}_l$ playing the part of the forms $v_l$).   \hfill $\Box$

\vspace{3ex}

An immediate (and new to our knowledge) consequence of this discussion is the analogue on every page $E_r$ of the Fr\"olicher spectral sequence of the classical {\bf Serre duality}. The well-definedness was proved in Corollary \ref{Cor:well-definedness_Er_duality}. The case $r=1$ is the Serre duality, while the case $r=2$ was proved in Theorem \ref{The:duality_E2}.

\begin{Cor}\label{Cor:Er_duality} Let $X$ be a compact complex manifold with $\mbox{dim}_\C X=n$. For every $r\in\N_{>0}$ and every $p,q\in\{0,\dots , n\}$, the canonical bilinear pairing  \begin{equation*} E_r^{p,\,q}(X)\times E_r^{n-p,\,n-q}(X)\longrightarrow\C, \hspace{3ex} (\{\alpha\}_{E_r},\,\{\beta\}_{E_r})\mapsto \int\limits_X\alpha\wedge\beta,\end{equation*} \noindent is {\bf non-degenerate}.

  \end{Cor}

\noindent {\it Proof.} Let $\{\alpha\}_{E_r}\in E_r^{p,\,q}(X)\setminus\{0\}$. If we fix an arbitrary Hermitian metric $\omega$ on $X$, we know from Proposition \ref{Prop:H_r_construction} that the associated harmonic space ${\cal H}_r^{p,\,q}$ is isomorphic to $E_r^{p,\,q}(X)$ and that the class $\{\alpha\}_{E_r}$ contains a (unique) representative $\alpha$ lying in ${\cal H}_r^{p,\,q}$. By Proposition \ref{Prop:E_r-closedness-star-harmonicity}, this is equivalent to $\alpha$ being both $E_r$-closed and $E_r^\star$-closed, while by Corollary \ref{Cor:E_r_duality}, this is further equivalent to $\star\bar\alpha$ being both $E_r^\star$-closed and $E_r$-closed, hence to $\star\bar\alpha$ lying in ${\cal H}_r^{n-p,\,n-q}$.

 In particular, $\star\bar\alpha$ represents a non-zero class $\{\star\bar\alpha\}_{E_r}\in E_r^{n-p,\,n-q}(X)$. We have $$(\{\alpha\}_{E_r},\,\{\star\bar\alpha\}_{E_r})\mapsto \int\limits_X\alpha\wedge\star\bar\alpha = ||\alpha||^2>0,$$

\noindent where $||\,\,\,||$ stands for the $L^2_\omega$-norm. This shows that for every non-zero class $\{\alpha\}_{E_r}\in E_r^{p,\,q}(X)$, the map $(\{\alpha\}_{E_r},\,\cdot):E_r^{n-p,\,n-q}(X)\longrightarrow\C$ does not vanish identically, proving the non-degeneracy of the pairing.  \hfill $\Box$

\begin{Rem}\label{Rem:alternative-approaches_E_r-duality} The numerical version of Corollary \ref{Cor:Er_duality} was proved in [Pop17] and the present version and its method of proof were already announced and used in various works by the first and third named authors. See, e.g. [BP18, $\S.3.4$]. There is no purely algebraic proof for these kind of results. Using some algebraic machinery, one could limit the analytic input to the classical Serre duality (see [Mil20] and [Ste20], the latter also applying to higher-page Bott-Chern and Aeppli cohomology). The harmonic methods used here and in the next section give some finer information besides the duality, which is interesting in its own right and useful in applications (see e.g. Lemma \ref{Lem:E_rA_pos-cones}).
\end{Rem}

\section{Higher-page Bott-Chern and Aeppli cohomologies: definition, Hodge theory and duality}\label{section:higher-page_BC-A} Let $X$ be an $n$-dimensional compact complex manifold. Fix an arbitrary positive integer $r$ and a bidegree $(p,\,q)$. In $\S.$\ref{subsection:duality_Er}, we defined the notions of $E_r$-closedness and $E_r$-exactness for forms $\alpha\in C^\infty_{p,\,q}(X)$ as higher-page analogues of $\bar\partial$-closedness (that can now be called $E_1$-closedness) and $\bar\partial$-exactness (that can now be called $E_1$-exactness). We then gave these notions explicit descriptions in Proposition \ref{Prop:E_r-closed-exact_conditions}.

In the same vein, we say that $\alpha$ is {\it $\overline{E}_r$-closed} if $\bar\alpha$ is $E_r$-closed and we say that $\alpha$ is {\it $\overline{E}_r$-exact} if $\bar\alpha$ is $E_r$-exact.  In particular, characterisations of $\overline{E}_r$-closedness and $\overline{E}_r$-exactness are obtained by permuting $\partial$ and $\bar\partial$ in the characterisations of $E_r$-closedness and $E_r$-exactness of Proposition \ref{Prop:E_r-closed-exact_conditions}.

Moreover, we can take our cue from Proposition \ref{Prop:E_r-closed-exact_conditions} to define higher-page analogues of $\partial\bar\partial$-closedness and $\partial\bar\partial$-exactness in the following way.

\begin{Def}\label{Def:E_rE_r-bar} Suppose that $r\geq 2$.

\vspace{1ex}

(i)\, We say that a form $\alpha\in C^\infty_{p,\,q}(X)$ is {\bf $E_r\overline{E}_r$-closed} if there exist smooth forms $\eta_1,\dots , \eta_{r-1}$ and $\rho_1,\dots , \rho_{r-1}$ such that the following two towers of $r-1$ equations are satisfied: \begin{align*}\label{eqn:towers_E_rE_r-bar-closedness} \partial\alpha & = \bar\partial\eta_1 &    \bar\partial\alpha & = \partial\rho_1 & \\
    \partial\eta_1 & = \bar\partial\eta_2 &   \bar\partial\rho_1 & = \partial\rho_2 & \\
     \vdots & & \\
     \partial\eta_{r-2} & =  \bar\partial\eta_{r-1},  &  \bar\partial\rho_{r-2} & = \partial\rho_{r-1}. &\end{align*}

\vspace{1ex}

\vspace{1ex}

(ii)\, We refer to the properties of $\alpha$ in the two towers of $(r-1)$ equations under (i) by saying that $\partial\alpha$, resp. $\bar\partial\alpha$, {\bf runs at least $(r-1)$ times}.

(iii)\, We say that a form $\alpha\in C^\infty_{p,\,q}(X)$ is {\bf $E_r\overline{E}_r$-exact} if there exist smooth forms $\zeta, \xi, \eta$ such that \begin{equation}\label{eqn:main-eq_E_rE_r-bar-exactness}\alpha = \partial\zeta + \partial\bar\partial\xi + \bar\partial\eta\end{equation}

\noindent and such that $\zeta$ and $\eta$ further satisfy the following conditions. There exist smooth forms $v_{r-3},\dots , v_0$ and $u_{r-3},\dots , u_0$ such that the following two towers of $r-1$ equations are satisfied: \begin{align*}\label{eqn:towers_E_rE_r-bar-exactness} \bar\partial\zeta & = \partial v_{r-3} &    \partial\eta & = \bar\partial u_{r-3} & \\
     \bar\partial v_{r-3} & = \partial v_{r-4} &   \partial u_{r-3} & = \bar\partial u_{r-4} & \\
     \vdots & & \\
     \bar\partial v_0 & =  0,  &  \partial u_0 & = 0. &\end{align*}

\vspace{1ex}

(iv)\, We refer to the properties of $\zeta$, resp. $\eta$, in the two towers of $(r-1)$ equations under (iii) by saying that $\bar\partial\zeta$, resp. $\partial\eta$, {\bf reaches $0$ in at most $(r-1)$ steps}. 

\end{Def}

 When $r-1=1$, the properties of $\bar\partial\zeta$, resp. $\partial\eta$, {\it reaching $0$ in $(r-1)$ steps} translate to $\bar\partial\zeta=0$, resp. $\partial\eta=0$.

To unify the definitions, we will also say that a form $\alpha\in C^\infty_{p,\,q}(X)$ is {\bf $E_1\overline{E}_1$-closed} (resp. {\bf $E_1\overline{E}_1$-exact}) if $\alpha$ is $\partial\bar\partial$-closed (resp. $\partial\bar\partial$-exact).

As with $E_r$ and $\overline{E}_r$, it follows at once from Definition \ref{Def:E_rE_r-bar} that the $E_r\overline{E}_r$-closedness condition becomes stronger and stronger as $r$ increases, while the $E_r\overline{E}_r$-exactness condition becomes weaker and weaker as $r$ increases. In other words, the following inclusions of vector spaces hold:

\vspace{2ex}

\noindent $\{\partial\bar\partial\mbox{-exact forms}\}\subset\dots\subset\{E_r\overline{E}_r\mbox{-exact forms}\}\subset\{E_{r+1}\overline{E}_{r+1}\mbox{-exact forms}\}\subset\dots $

\vspace{1ex}

\hfill $\dots\subset\{E_{r+1}\overline{E}_{r+1}\mbox{-closed forms}\}\subset\{E_r\overline{E}_r\mbox{-closed forms}\}\subset\dots\subset\{\partial\bar\partial\mbox{-closed forms}\}.$

\vspace{3ex}

The following statement collects a few other immediate relations among these notions.

\begin{Lem}\label{Lem:E_rE_r-bar-properties} Fix an arbitrary $r\in\N_{>0}$.

  \vspace{1ex}

  (i)\, A pure-type form $\alpha$ is simultaneously $E_r$-closed and $\overline{E}_r$-closed if and only if $\alpha$ is simultaneously $\partial$-closed and $\bar\partial$-closed. This is further equivalent to $\alpha$ being $d$-closed.

  \vspace{1ex}
  
  (ii)\, If $\alpha$ is $E_r\overline{E}_r$-exact, then each of the classes $\{\alpha\}_{E_r}$ and $\{\alpha\}_{\overline{E}_r}$ contains a $\partial\bar\partial$-exact form and $\alpha$ is both $E_r$-exact and $\overline{E}_r$-exact.

  \vspace{1ex}

  (iii)\, Fix any bidegree $(p,\,q)$ and let $\alpha\in C^\infty_{p,\,q}(X)$. If $\alpha$ is $E_r\overline{E}_r$-exact for some $r\in\N_{>0}$, then $\alpha$ is $d$-exact.

\end{Lem}

\noindent {\it Proof.} (i) is obvious. To see (ii), let $\alpha = \partial\zeta + \partial\bar\partial\xi + \bar\partial\eta$ be $E_r\overline{E}_r$-exact, with $\zeta$ and $\eta$ satisfying the conditions under (ii) of Definition \ref{Def:E_rE_r-bar}. Then \[\{\alpha\}_{E_r} =\{\alpha-\partial\zeta-\bar\partial\eta\}_{E_r}=\{\partial\bar\partial\xi\}_{E_r}\hspace{3ex} \mbox{and} \hspace{3ex} \{\alpha\}_{\overline{E}_r} =\{\alpha-\partial\zeta-\bar\partial\eta\}_{\overline{E}_r}=\{\partial\bar\partial\xi\}_{\overline{E}_r},\]

\noindent while $\alpha = \partial\zeta + \bar\partial(-\partial\xi + \eta)$ is $E_r$-exact and $\alpha = \partial(\zeta + \bar\partial\xi) + \bar\partial\eta$ is $\overline{E}_r$-exact.

To prove (iii), let $\alpha = \partial\zeta + \partial\bar\partial\xi + \bar\partial\eta$, where $\zeta$ and $\eta$ satisfy the conditions in the two towers under (ii) of Definition \ref{Def:E_rE_r-bar}. Going down the first tower, we get \[\partial\zeta = d\zeta - \bar\partial\zeta = d\zeta - \partial v_{r-3} = d(\zeta - v_{r-3}) + \partial v_{r-4} = \dots = d(\zeta - v_{r-3} + \dots + (-1)^r\,v_0).\]

\noindent In particular, $\partial\zeta$ is $d$-exact.

Similarly, going down the second tower, we get \[\bar\partial\eta = d(\eta - u_{r-3} + \dots + (-1)^r\,u_0).\]

\noindent In particular, $\bar\partial\eta$ is $d$-exact.

Since $\partial\bar\partial\xi$ is also $d$-exact, we infer that $\alpha$ is $d$-exact. Explicitly, we have \[\alpha = d[(\zeta + \eta) + \bar\partial\xi - w_{r-3} + \dots + (-1)^r\,w_0],\]

\noindent where $w_j:=u_j+v_j$ for all $j$.  \hfill $\Box$

\vspace{3ex}

The main takeaway from Lemma \ref{Lem:E_rE_r-bar-properties} is that $E_r\overline{E}_r$-exactness implies $E_r$-exactness, $\overline{E}_r$-exactness and $d$-exactness. Let us now pause briefly to notice a property involving the spaces ${\cal C}^{p,\,q}_r$ of $E_r$-exact $(p,\,q)$-forms, resp. $\overline{\cal C}^{p,\,q}_r$ of $\overline{E}_r$-exact $(p,\,q)$-forms.

\begin{Lem}\label{Lem:E_r-exact_E_r-bar-exact_sum} Fix an arbitrary $r\in\N_{>0}$. For any bidegree $(p,\,q)$, the following identity of vector subspaces of $C^\infty_{p,\,q}(X)$ holds: \[{\cal C}^{p,\,q}_r + \overline{\cal C}^{p,\,q}_r = \mbox{Im}\,\partial + \mbox{Im}\,\bar\partial.\]

\end{Lem}

\noindent {\it Proof.} For any bidegree $(p,\,q)$, consider the vector spaces (see (iv) of Definition \ref{Def:E_rE_r-bar} for the terminology): \begin{eqnarray*}{\cal E}^{p,\,q}_{\partial,\,r} &:= & \{\alpha\in C^\infty_{p,\,q}(X)\,\mid\,\partial\alpha\hspace{1ex}\mbox{reaches 0 in at most r steps}\}, \\
 {\cal E}^{p,\,q}_{\bar\partial,\,r} & := & \{\beta\in C^\infty_{p,\,q}(X)\,\mid\,\bar\partial\beta\hspace{1ex}\mbox{reaches 0 in at most r steps}\}.\end{eqnarray*}

From the definitions, we get: ${\cal C}^{p,\,q}_r=\partial({\cal E}^{p,\,q}_{\bar\partial,\,{r-1}}) + \mbox{Im}\,\bar\partial$ and $\overline{\cal C}^{p,\,q}_r = \mbox{Im}\,\partial + \bar\partial({\cal E}^{p,\,q}_{\partial,\,{r-1}})$. This trivially implies the contention.  \hfill $\Box$

\vspace{3ex}

We now come to the main definitions of this subsection.

\begin{Def}\label{Def:E_r-BC_E_r-A} Let $X$ be an $n$-dimensional compact complex manifold. Fix $r\in\N_{>0}$ and a bidegree $(p,\,q)$.

\vspace{1ex}

(i)\, The {\bf $E_r$-Bott-Chern} cohomology group of bidegree $(p,\,q)$ of $X$ is defined as the following quotient complex vector space: \[E_{r,\,BC}^{p,\,q}(X):=\frac{\{\alpha\in C^\infty_{p,\,q}(X)\,\mid\,d\alpha=0\}}{\{\alpha\in C^\infty_{p,\,q}(X)\,\mid\,\alpha\hspace{1ex}\mbox{is}\hspace{1ex} E_r\overline{E}_r\mbox{-exact}\}}.\]

\vspace{1ex}

(ii)\, The {\bf $E_r$-Aeppli} cohomology group of bidegree $(p,\,q)$ of $X$ is defined as the following quotient complex vector space: \[E_{r,\,A}^{p,\,q}(X):=\frac{\{\alpha\in C^\infty_{p,\,q}(X)\,\mid\,\alpha\hspace{1ex}\mbox{is}\hspace{1ex} E_r\overline{E}_r-\mbox{closed}\}}{\{\alpha\in C^\infty_{p,\,q}(X)\,\mid\,\alpha\in\mbox{Im}\,\partial + \mbox{Im}\,\bar\partial\}}.\]

\end{Def}

When $r=1$, the above groups coincide with the standard Bott-Chern, respectively Aeppli, cohomology groups (see [BC65] and [Aep62]). Note that, by (i) of Lemma \ref{Lem:E_rE_r-bar-properties}, the representatives of $E_r$-Bott-Chern classes can be alternatively described as the forms that are simultaneously $E_r$-closed and $\overline{E}_r$-closed, while by Lemma \ref{Lem:E_r-exact_E_r-bar-exact_sum}, the $E_r$-Aeppli-exact forms can be alternatively described as those forms lying in ${\cal C}^{p,\,q}_r + \overline{\cal C}^{p,\,q}_r$.

Also note that the inclusions of vector spaces spelt out just before Lemma \ref{Lem:E_rE_r-bar-properties} and their analogues for the $E_r$- and $\overline{E}_r$-cohomologies imply the following inequalities of dimensions: \[\dots\leq\mbox{dim}\,E_{r,\,BC}^{p,\,q}(X)\leq\mbox{dim}\,E_{r-1,\,BC}^{p,\,q}(X)\leq\dots\leq\mbox{dim}\,E_{1,\,BC}^{p,\,q}(X)=\mbox{dim}\,H_{BC}^{p,\,q}(X)\]

\noindent and their analogues for the $E_r$-Aeppli cohomology spaces.

The first step towards extending to the higher pages of the Fr\"olicher spectral sequence the standard Serre-type duality between the classical Bott-Chern and Aeppli cohomology groups of complementary bidegrees is the following

\begin{Prop}\label{Prop:E_r_BC-A_duality1} Let $X$ be a compact complex manifold with $\mbox{dim}_\C X=n$. For every $r\in\N_{>0}$ and all $p,q\in\{0,\dots , n\}$, the following bilinear pairing is {\bf well defined}: \[E_{r,\,BC}^{p,\,q}(X)\times E_{r,\,A}^{n-p,\,n-q}(X)\longrightarrow\C, \hspace{3ex} \bigg(\{\alpha\}_{E_{r,\,BC}},\,\{\beta\}_{E_{r,\,A}}\bigg)\mapsto\int\limits_X\alpha\wedge\beta,\]

\noindent in the sense that it is independent of the choice of representative of either of the classes $\{\alpha\}_{E_{r,\,BC}}$ and $\{\beta\}_{E_{r,\,A}}$.

\end{Prop}

\noindent {\it Proof.} The proof consists in a series of integrations by parts (mathematical ping-pong). 

\vspace{1ex}

$\bullet$ To prove independence of the choice of representative of the $E_r$-Bott-Chern class, let us modify a representative $\alpha$ to some representative $\alpha + \partial\zeta + \partial\bar\partial\xi + \bar\partial\eta$ of the same $E_r$-Bott-Chern class. This means that $\partial\zeta + \partial\bar\partial\xi + \bar\partial\eta$ is $E_r\overline{E}_r$-exact, so $\zeta$ and $\eta$ satisfy the towers of $r-1$ equations under (ii) of Definition \ref{Def:E_rE_r-bar}. We have \begin{eqnarray*}\int\limits_X(\alpha + \partial\zeta + \partial\bar\partial\xi + \bar\partial\eta)\wedge\beta & = & \int\limits_X\alpha\wedge\beta \pm\int\limits_X\zeta\wedge\partial\beta \pm \int\limits_X\xi\wedge\partial\bar\partial\beta \pm  \int\limits_X\eta\wedge\bar\partial\beta.\end{eqnarray*} 

Since $\beta$ is $E_r\overline{E}_r$-closed, it is also $\partial\bar\partial$-closed (see (i) of Lemma \ref{Lem:E_rE_r-bar-properties}), so the last but one integral above vanishes. 

Using the $r-1$ equations in the first tower under (i) of Definition \ref{Def:E_rE_r-bar} (with $\beta$ in place of $\alpha$) and the first tower under (ii) of the same definition, we get: \begin{eqnarray*}\int\limits_X\zeta\wedge\partial\beta & = & \int\limits_X\zeta\wedge\bar\partial\eta_1 = \pm\int\limits_X\bar\partial\zeta\wedge\eta_1 = \pm\int\limits_X\partial v_{r-3}\wedge\eta_1 = \pm\int\limits_X v_{r-3}\wedge\partial\eta_1 \\
 & = & \pm\int\limits_X v_{r-3}\wedge\bar\partial\eta_2 = \pm\int\limits_X \bar\partial v_{r-3}\wedge\eta_2 = \pm\int\limits_X\partial v_{r-4}\wedge\eta_2 = \pm\int\limits_X v_{r-4}\wedge\partial\eta_2  \\
 & \vdots & \\
 & = & \pm\int\limits_X v_0\wedge\bar\partial\eta_{r-1} = \pm\int\limits_X \bar\partial v_0\wedge\eta_{r-1} =0,\end{eqnarray*}

\noindent where the last identity follows from $\bar\partial v_0=0$. 

Playing the analogous mathematical ping-pong while using the second tower under both (i) and (ii) of Definition \ref{Def:E_rE_r-bar}, we get: \begin{eqnarray*}\int\limits_X\eta\wedge\bar\partial\beta & = & \int\limits_X\eta\wedge\partial\rho_1 = \pm\int\limits_X\partial\eta\wedge\rho_1 =\pm\int\limits_X\bar\partial u_{r-3}\wedge\rho_1 = \pm\int\limits_X u_{r-3}\wedge\bar\partial\rho_1 \\
 & = & \pm\int\limits_X u_{r-3}\wedge\partial\rho_2 = \pm\int\limits_X\partial u_{r-3}\wedge\rho_2 = \pm\int\limits_X\bar\partial u_{r-4}\wedge\rho_2 = \pm\int\limits_X u_{r-4}\wedge\bar\partial\rho_2   \\
 & \vdots & \\
 & = & \pm\int\limits_X u_0\wedge\partial\rho_{r-1} = \pm\int\limits_X \partial u_0\wedge\rho_{r-1} =0,\end{eqnarray*}     

 where the last identity follows from $\partial u_0=0$.

We conclude that $\int_X(\alpha + \partial\zeta + \partial\bar\partial\xi + \bar\partial\eta)\wedge\beta = \int_X\alpha\wedge\beta$.

\vspace{1ex}
 
$\bullet$ To prove independence of the choice of representative of the $E_r$-Aeppli class, let us modify a representative $\beta$ to some representative $\beta + \partial\zeta + \bar\partial\xi$ of the same $E_r$-Aeppli class. So, $\zeta$ and $\xi$ are arbitrary forms. We get: \begin{eqnarray*}\int\limits_X\alpha\wedge(\beta + \partial\zeta + \bar\partial\xi) & = & \int\limits_X\alpha\wedge\beta \pm \int\limits_X\partial\alpha\wedge\zeta \pm \int\limits_X\bar\partial\alpha\wedge\xi =0,\end{eqnarray*}  

\noindent where the last identity follows from $\partial\alpha=0$ and $\bar\partial\alpha=0$.   \hfill $\Box$

\vspace{2ex}

We now take up the issue of the non-degeneracy of the above bilinear pairing. For the sake of expediency, we start by defining the dual notion to the $E_r\overline{E}_r$-closedness of Definition \ref{Def:E_rE_r-bar} after we have fixed a metric.

\begin{Def}\label{Def:E_r_star_E_r-bar_star} Let $(X,\,\omega)$ be a compact complex Hermitian manifold. Fix an integer $r\geq 2$ and a bidegree $(p,\,q)$.

 We say that a form $\alpha\in C^\infty_{p,\,q}(X)$ is {\bf $E_r^\star\overline{E}_r^\star$-closed} with respect to the Hermitian metric $\omega$ if there exist smooth forms $a_1,\dots , a_{r-1}$ and $b_1,\dots , b_{r-1}$ such that the following two towers of $r-1$ equations are satisfied: \begin{align*} \partial^\star\alpha & = \bar\partial^\star a_1 &    \bar\partial^\star\alpha & = \partial^\star b_1 & \\
     \partial^\star a_1 & = \bar\partial^\star a_2 &   \bar\partial^\star b_1 & = \partial^\star b_2 & \\
     \vdots & & \\
     \partial^\star a_{r-2} & =  \bar\partial^\star a_{r-1},  &  \bar\partial^\star b_{r-2} & = \partial^\star b_{r-1}. & \end{align*}

\end{Def}

That this notion is indeed dual to the $E_r\overline{E}_r$-closedness via the Hodge star operator $\star=\star_\omega$ associated with the metric $\omega$ is the content of the following analogue of Corollary \ref{Cor:E_r_duality}.

\begin{Lem}\label{Lem:E_rE_r-bar_duality} In the setting of Definition \ref{Def:E_r_star_E_r-bar_star}, the following equivalence holds for every form $\alpha\in C^\infty_{p,\,q}(X)$: \\

\hspace{25ex} $\alpha$ is $E_r\overline{E}_r$-closed $\iff$ $\star\bar\alpha$ is $E_r^\star\overline{E}_r^\star$-closed.

\end{Lem}

\noindent {\it Proof.} Thanks to conjugations, to the fact that $\star\star =\pm\, \mbox{Id}$ (with the sign depending on the parity of the total degree of the forms involved) and to $\star$ being an isomorphism, the two towers of $r-1$ equations that express the $E_r\overline{E}_r$-closedness of $\alpha$ (cf. (i) of Definition \ref{Def:E_rE_r-bar}) translate to \begin{align*} (-\star\bar\partial\star)(\star\bar\alpha) & = (-\star\partial\star)(\star\bar\eta_1) &    (-\star\partial\star)(\star\bar\alpha) & = (-\star\bar\partial\star)(\star\bar\rho_1) & \\
     (-\star\bar\partial\star)(\star\bar\eta_1) & = (-\star\partial\star)(\star\bar\eta_2) &   (-\star\partial\star)(\star\bar\rho_1) & = (-\star\bar\partial\star)(\star\bar\rho_2) & \\
     \vdots & & \\
     (-\star\bar\partial\star)(\star\bar\eta_{r-2}) & =  (-\star\partial\star)(\star\bar\eta_{r-1}),  &  (-\star\partial\star)(\star\bar\rho_{r-2}) & = (-\star\bar\partial\star)(\star\bar\rho_{r-1}). & \end{align*}

\noindent Now, put $a_j:=\star\bar\eta_j$ and $b_j:=\star\bar\rho_j$ for all $j\in\{1,\dots , r-1\} $. Since $-\star\bar\partial\star = \partial^\star$ and $-\star\partial\star = \bar\partial^\star$, these two towers amount to $\star\bar\alpha$ being $E_r^\star\overline{E}_r^\star$-closed. (See Definition \ref{Def:E_r_star_E_r-bar_star}).  \hfill $\Box$

\vspace{2ex}

We now come to two crucial lemmas from which Hodge isomorphisms for the $E_r$-Bott-Chern and the $E_r$-Aeppli cohomologies will follow. Based on the terminology introduced in (ii) of Definition \ref{Def:E_rE_r-bar}, we define the vector spaces: \begin{eqnarray*}{\cal F}^{p,\,q}_{\partial,\,r} &:= & \{\alpha\in C^\infty_{p,\,q}(X)\,\mid\,\partial\alpha\hspace{1ex}\mbox{runs at least r times}\}, \\
 {\cal F}^{p,\,q}_{\bar\partial,\,r} & := & \{\beta\in C^\infty_{p,\,q}(X)\,\mid\,\bar\partial\beta\hspace{1ex}\mbox{runs at least r times}\}\end{eqnarray*}

\noindent and their analogues ${\cal F}^{p,\,q}_{\partial^\star,\,r}$ and ${\cal F}^{p,\,q}_{\bar\partial^\star,\,r}$ when $\partial$ is replaced by $\partial^\star$ and $\bar\partial$ is replaced by $\bar\partial^\star$. Note that the space of $E_r^\star\overline{E}_r^\star$-closed $(p,\,q)$-forms defined in Definition \ref{Def:E_r_star_E_r-bar_star} is precisely the intersection ${\cal F}^{p,\,q}_{\partial^\star,\,r-1}\cap{\cal F}^{p,\,q}_{\bar\partial^\star,\,r-1}$.

\begin{Lem}\label{Lem:star-duality_1} Let $(X,\,\omega)$ be a compact complex Hermitian manifold. Fix an integer $r\geq 2$, a bidegree $(p,\,q)$ and a form $\alpha\in C^\infty_{p,\,q}(X)$.

  \vspace{1ex}

  The following two statements are {\bf equivalent}.

  \vspace{1ex}

  (i)\, $\alpha$ is $E_r^\star\overline{E}_r^\star$-closed (w.r.t. $\omega$);

  \vspace{1ex}

  (ii)\, $\alpha$ is $L^2_\omega$-orthogonal to the space of smooth $E_r\overline{E}_r$-exact $(p,\,q)$-forms.

\end{Lem}

\noindent {\it Proof.} ``(i)$\implies$(ii)''\, Suppose that $\alpha$ is $E_r^\star\overline{E}_r^\star$-closed. This means that $\alpha$ satisfies the two towers of $(r-1)$ equations in Definition \ref{Def:E_r_star_E_r-bar_star}. Let $\beta=\partial\zeta + \partial\bar\partial\xi + \bar\partial\eta$ be an arbitrary $E_r\overline{E}_r$-exact $(p,\,q)$-form. So, $\zeta$ and $\eta$ satisfy the respective towers of $r-1$ equations under (ii) of Definition \ref{Def:E_rE_r-bar}. For the $L^2_\omega$-inner product of $\alpha$ and $\beta$, we get: \begin{eqnarray}\label{eqn:alpha-beta_1}\langle\langle\alpha,\,\beta\rangle\rangle = \langle\langle\partial^\star\alpha,\,\zeta\rangle\rangle + \langle\langle\bar\partial^\star\partial^\star\alpha,\,\xi\rangle\rangle + \langle\langle\bar\partial^\star\alpha,\,\eta\rangle\rangle.\end{eqnarray}

Since $\bar\partial^\star\partial^\star\alpha = \bar\partial^\star\bar\partial^\star a_1=0$, the middle term on the r.h.s. of (\ref{eqn:alpha-beta_1}) vanishes.

For the first term on the r.h.s. of (\ref{eqn:alpha-beta_1}), we use the towers of equations satisfied by $\alpha$ and $\zeta$ to get: \begin{eqnarray*}\langle\langle\partial^\star\alpha,\,\zeta\rangle\rangle & = & \langle\langle\bar\partial^\star a_1,\,\zeta\rangle\rangle = \langle\langle a_1,\,\bar\partial\zeta\rangle\rangle = \langle\langle a_1,\,\partial v_{r-3}\rangle\rangle =  \langle\langle \partial^\star a_1,\, v_{r-3}\rangle\rangle = \langle\langle \bar\partial^\star a_2,\, v_{r-3}\rangle\rangle \\
   & = & \langle\langle a_2,\, \bar\partial v_{r-3}\rangle\rangle = \langle\langle a_2,\, \partial v_{r-4}\rangle\rangle  \\
   & \vdots &  \\
   & = & \langle\langle a_{r-1},\, \bar\partial v_0\rangle\rangle =0,\end{eqnarray*}

\noindent where the last identity followed from the property $\bar\partial v_0=0$.

For the last term on the r.h.s. of (\ref{eqn:alpha-beta_1}), we use the towers of equations satisfied by $\alpha$ and $\eta$ to get: \begin{eqnarray*}\langle\langle\bar\partial^\star\alpha,\,\eta\rangle\rangle & = & \langle\langle\partial^\star b_1,\,\eta\rangle\rangle = \langle\langle b_1,\,\partial\eta\rangle\rangle = \langle\langle b_1,\,\bar\partial u_{r-3}\rangle\rangle =  \langle\langle \bar\partial^\star b_1,\, u_{r-3}\rangle\rangle = \langle\langle \partial^\star b_2,\, u_{r-3}\rangle\rangle \\
   & = & \langle\langle b_2,\, \partial u_{r-3}\rangle\rangle = \langle\langle b_2,\, \bar\partial u_{r-4}\rangle\rangle  \\
   & \vdots &  \\
   & = & \langle\langle b_{r-1},\, \partial u_0\rangle\rangle =0,\end{eqnarray*}

\noindent where the last identity followed from the property $\partial u_0=0$.

``(ii)$\implies$(i)''\, Suppose that $\alpha$ is orthogonal to all the smooth $E_r\overline{E}_r$-exact $(p,\,q)$-forms $\beta$. These forms are of the shape $\beta = \partial\zeta + \partial\bar\partial\xi + \bar\partial\eta$, where $\xi$ is subject to no condition, while $\zeta\in{\cal E}_{\bar\partial,\,r-1}^{p-1,\,q}$ and $\eta\in{\cal E}_{\partial,\,r-1}^{p,\,q-1}$. (See notation introduced in the proof of Lemma \ref{Lem:E_r-exact_E_r-bar-exact_sum}).

 The orthogonality condition is equivalent to the following three identities: \[(a)\, \langle\langle \bar\partial^\star\partial^\star\alpha,\, \xi\rangle\rangle =0,  \hspace{3ex} (b)\, \langle\langle\partial^\star\alpha,\, \zeta\rangle\rangle =0,  \hspace{3ex} (c)\, \langle\langle\bar\partial^\star\alpha,\, \eta\rangle\rangle =0\]

\noindent holding for all the forms $\zeta, \xi, \eta$ satisfying the above conditions.

Since $\xi$ is subject to no condition, (a) amounts to $\bar\partial^\star\partial^\star\alpha=0$. This means that $\partial^\star\alpha\in\ker\bar\partial^\star$ and $\bar\partial^\star\alpha\ker\partial^\star$. Condition (b) requires $\partial^\star\alpha\perp{\cal E}_{\bar\partial,\,r-1}^{p-1,\,q}$, while (c) requires $\bar\partial^\star\alpha\perp{\cal E}_{\partial,\,r-1}^{p,\,q-1}$.

\vspace{1ex}

$\bullet$ {\it Unravelling condition (b).} The forms $\zeta\in{\cal E}_{\bar\partial,\,r-1}^{p-1,\,q}$ are characterised by the existence of forms $v_{r-3},\dots , v_0$ satisfying the first tower of $(r-1)$ equations in (iii) of Definition \ref{Def:E_rE_r-bar}. That tower imposes the condition $v_{r-j}\in{\cal E}_{\bar\partial,\,r-j+1}\cap{\cal F}_{\partial,\,j-2}$ for every $j\in\{3,\dots , j\}$. (We have dropped the superscripts to lighten the notation.)

Now, every form $\zeta\in\ker\Delta''$ satisfies the condition $\bar\partial\zeta=0$, hence $\zeta\in{\cal E}_{\bar\partial,\,1}\subset{\cal E}_{\bar\partial,\,r-1}$. From condition (b), we get $\partial^\star\alpha\perp\ker\Delta''$. Since $\ker\bar\partial^\star$ (to which $\partial^\star\alpha$ belongs by condition (a)) is the orthogonal direct sum between $\ker\Delta''$ and $\mbox{Im}\,\bar\partial^\star$, we get $\partial^\star\alpha\in\mbox{Im}\,\bar\partial^\star$, so

\begin{equation} 
\label{eqn:proof-ortho1} \partial^\star\alpha=\bar\partial^\star a_1
\end{equation}

\noindent for some form $a_1$. Condition (b) becomes:\[0=\langle\langle \partial^\star\alpha,\, \zeta\rangle\rangle = \langle\langle a_1,\, \bar\partial\zeta\rangle\rangle = \langle\langle a_1,\, \partial v_{r-3}\rangle\rangle = \langle\langle\partial^\star a_1,\, v_{r-3}\rangle\rangle \hspace{2ex} \mbox{for all} \hspace{2ex} v_{r-3}\in{\cal E}_{\bar\partial,\,r-2}\cap{\cal F}_{\partial,\,1}.\]

\noindent In other words, $\partial^\star a_1\perp({\cal E}_{\bar\partial,\,r-2}\cap{\cal F}_{\partial,\,1})$.

We will now use the $3$-space decomposition (\ref{eqn:appendix_4-space_decomp_E-F}) of $C^\infty_{p,\,q}(X)$ for the case $r=2$. (See Proposition \ref{Prop:appendix_3-space_decomp_E-F} in Appendix one.) It is immediate to check the inclusion ${\cal E}_{\bar\partial,\,r-2}\cap{\cal F}_{\partial,\,1}\supset{\cal H}_2\oplus(\mbox{Im}\,\bar\partial + \partial({\cal E}_{\bar\partial,\,1}))$. Therefore, condition (b) implies that $\partial^\star a_1\perp({\cal H}_2\oplus(\mbox{Im}\,\bar\partial + \partial({\cal E}_{\bar\partial,\,1})))$. Since the orthogonal complement of ${\cal H}_2\oplus(\mbox{Im}\,\bar\partial + \partial({\cal E}_{\bar\partial,\,1}))$ is $\partial^\star({\cal E}_{\bar\partial^\star,\,1}) + \mbox{Im}\,\bar\partial^\star$ by the $3$-space decomposition (\ref{eqn:appendix_4-space_decomp_E-F}) for $r=2$, we infer that $\partial^\star a_1\in\partial^\star({\cal E}_{\bar\partial^\star,\,1}) + \mbox{Im}\,\bar\partial^\star$. Therefore, there exist forms $b_1\in\ker\bar\partial^\star$ and $a_2$ such that \begin{equation}
\label{eqn:proof-ortho2}\partial^\star a_1 = \partial^\star b_1 + \bar\partial^\star a_2.
\end{equation}

\noindent Since $\bar\partial^\star b_1=0$, equations (\ref{eqn:proof-ortho1}) and (\ref{eqn:proof-ortho2}) yield: \begin{equation}\label{eqn:proof_ortho3}\begin{array}{rcl} \partial^\star\alpha&=&\bar\partial^\star(a_1-b_1)  \\
\partial^\star(a_1 - b_1)&=&\bar\partial^\star a_2.\end{array} \end{equation}

Thus, condition  (b) becomes: \begin{eqnarray*} 0 & = &\langle\langle \partial^\star\alpha,\, \zeta\rangle\rangle = \langle\langle\bar\partial^\star(a_1-b_1),\, \zeta\rangle\rangle = \langle\langle a_1-b_1,\, \partial v_{r-3}\rangle\rangle = \langle\langle\partial^\star(a_1-b_1),\, v_{r-3}\rangle\rangle \\
   & = & \langle\langle\bar\partial^\star a_2,\, v_{r-3}\rangle\rangle = \langle\langle a_2,\, \partial v_{r-4}\rangle\rangle = \langle\langle \partial^\star a_2,\, v_{r-4}\rangle\rangle \hspace{2ex} \mbox{for all} \hspace{2ex} v_{r-4}\in{\cal E}_{\bar\partial,\,r-3}\cap{\cal F}_{\partial,\,2}.\end{eqnarray*}

\noindent In other words, $\partial^\star a_2\perp({\cal E}_{\bar\partial,\,r-3}\cap{\cal F}_{\partial,\,2})$.

Now, it is immediate to check the inclusion ${\cal E}_{\bar\partial,\,r-3}\cap{\cal F}_{\partial,\,2}\supset{\cal H}_3\oplus(\mbox{Im}\,\bar\partial + \partial({\cal E}_{\bar\partial,\,2}))$. Since the orthogonal complement of ${\cal H}_3\oplus(\mbox{Im}\,\bar\partial + \partial({\cal E}_{\bar\partial,\,2}))$ is $\partial^\star({\cal E}_{\bar\partial^\star,\,2}) + \mbox{Im}\,\bar\partial^\star$ by the $3$-space decomposition (\ref{eqn:appendix_4-space_decomp_E-F}) for $r=3$, we infer that $\partial^\star a_2\in\partial^\star({\cal E}_{\bar\partial^\star,\,1}) + \mbox{Im}\,\bar\partial^\star$. Therefore, there exist forms $b_2\in{\cal E}_{\bar\partial^\star,\, 2}$ and $a_3$ such that 
\begin{equation}\label{eqn:proof-ortho4}\partial^\star a_2 = \partial^\star b_2 + \bar\partial^\star a_3.
\end{equation}

\noindent Since the condition $b_2\in{\cal E}_{\bar\partial^\star,\, 2}$ translates to the equations \begin{equation}
\bar\partial^\star b_2 = \partial^\star c_1  \hspace{3ex}  \mbox{and} \hspace{3ex} \bar\partial^\star c_1=0,
\end{equation}

\noindent for some form $c_1$, equations (\ref{eqn:proof_ortho3}) and (\ref{eqn:proof-ortho4}) yield: \begin{eqnarray*}\label{eqn:proof_ortho5} & & \partial^\star\alpha = \bar\partial^\star(a_1-b_1-c_1)  \\
  & & \partial^\star(a_1-b_1-c_1) = \bar\partial^\star(a_2-b_2) \\
  & & \partial^\star(a_2-b_2) = \bar\partial^\star a_3. \end{eqnarray*}

Continuing in this way, we inductively get the following tower of $(r-1)$ equations: \begin{equation}\label{eqn:proof_ortho6}\begin{array}{rcl}\partial^\star\alpha&=&\bar\partial^\star(a_1-b_1-c_1-c_1^{(3)}-\dots - c_1^{(r-2)})  \\
 \partial^\star(a_1-b_1-c_1-c_1^{(3)}-\dots - c_1^{(r-2)})&=&\bar\partial^\star(a_2-b_2-c_2^{(3)}-\dots -c_2^{(r-2)}) \\
 &\vdots& \\
  \partial^\star(a_{r-2}-b_{r-2})&=& \bar\partial^\star a_{r-1},\end{array}\end{equation}

\noindent where $b_j\in{\cal E}_{\bar\partial^\star,\, j}$ for all $j\in\{1,\dots , r-2\}$, so $b_j$ satisfies the following tower of $j$ equations: \begin{equation*}\begin{array}{rcl}\bar\partial^\star b_j &=& \partial^\star c_{j-1}^{(j)} \\
\bar\partial^\star c_{j-1}^{(j)} &=& \partial^\star c_{j-2}^{(j)} \\
 &\vdots& \\
\bar\partial^\star c_2^{(j)} &=& \partial^\star c_1^{(j)} \\
\bar\partial^\star c_1^{(j)} &=& 0, \end{array}\end{equation*}

\noindent for some forms $c_l^{(j)}$.

Consequently, conditions (a) and (b) to which $\alpha$ is subject imply that $\alpha\in{\cal F}_{\partial^\star,\,r-1}$ (cf. tower (\ref{eqn:proof_ortho6})), which is the first of the two conditions required for $\alpha$ to be $E_r^\star\overline{E}_r^\star$-closed under Definition \ref{Def:E_r_star_E_r-bar_star}.

\vspace{1ex}

$\bullet$ {\it Unravelling condition (c).} Proceeding in a similar fashion, with $\partial^\star$ and $\bar\partial^\star$ permuted, we infer that conditions (a) and (c) to which $\alpha$ is subject imply that $\alpha\in{\cal F}_{\bar\partial^\star,\,r-1}$, which is the second of the two conditions required for $\alpha$ to be $E_r^\star\overline{E}_r^\star$-closed under Definition \ref{Def:E_r_star_E_r-bar_star}. 

$\bullet$ We conclude that $\alpha$ is indeed $E_r^\star\overline{E}_r^\star$-closed.  \hfill $\Box$

\vspace{2ex}

The immediate consequence of Lemma \ref{Lem:star-duality_1} is the following {\it Hodge isomorphism} for the $E_r$-Bott-Chern cohomology.

\begin{Cor-Def}\label{Cor-Def:E_r-BC_Hodge-isom}Let $(X,\,\omega)$ be a compact complex Hermitian manifold. For every bidegree $(p,\,q)$ and every $r\in\N_{>0}$, every $E_r$-Bott-Chern cohomology class $\{\alpha\}_{E_{r,\,BC}}\in E_{r,\,BC}^{p,\,q}(X)$ can be represented by a unique form $\alpha\in C^\infty_{p,\,q}(X)$ satisfying the following three conditions: \\

  \hspace{20ex} $\alpha$ is $\partial$-closed, $\bar\partial$-closed and $E_r^\star\overline{E}_r^\star$-closed.

  \vspace{2ex}

  Any such form $\alpha$ is called {\bf $E_r$-Bott-Chern harmonic} with respect to the metric $\omega$.

  There is a vector-space {\bf isomorphism} depending on the metric $\omega$:

  \[E_{r,\,BC}^{p,\,q}(X)\simeq{\cal H}_{r,\,BC}^{p,\,q}(X),\]

  \noindent where ${\cal H}_{r,\,BC}^{p,\,q}(X)\subset C^\infty_{p,\,q}(X)$ is the space of $E_r$-Bott-Chern harmonic $(p,\,q)$-forms associated with $\omega$.

\end{Cor-Def}

Of course, the above isomorphism maps every class $\{\alpha\}_{E_{r,\,BC}}\in E_{r,\,BC}^{p,\,q}(X)$ to its unique $E_r$-Bott-Chern harmonic representative.

\vspace{3ex}

The analogous statement for the $E_r$-Aeppli cohomology follows at once from standard material. Indeed, it is classical that the $L^2_\omega$-orthogonal complement of $\mbox{Im}\,\partial$ (resp. $\mbox{Im}\,\bar\partial$) in $C^\infty_{p,\,q}(X)$ is $\ker\partial^\star$ (resp. $\ker\bar\partial^\star$). The immediate consequence of this is the following {\it Hodge isomorphism} for the $E_r$-Aeppli cohomology.

\begin{Cor-Def}\label{Cor-Def:E_r-A_Hodge-isom}Let $(X,\,\omega)$ be a compact complex Hermitian manifold. For every bidegree $(p,\,q)$, every $E_r$-Aeppli cohomology class $\{\alpha\}_{E_{2,\,A}}\in E_{r,\,A}^{p,\,q}(X)$ can be represented by a unique form $\alpha\in C^\infty_{p,\,q}(X)$ satisfying the following three conditions: \\

  \hspace{20ex} $\alpha$ is $E_r\overline{E}_r$-closed, $\partial^\star$-closed and $\bar\partial^\star$-closed.

  \vspace{2ex}

  Any such form $\alpha$ is called {\bf $E_r$-Aeppli harmonic} with respect to the metric $\omega$.

  There is a vector-space {\bf isomorphism} depending on the metric $\omega$:

  \[E_{r,\,A}^{p,\,q}(X)\simeq{\cal H}_{r,\,A}^{p,\,q}(X),\]

  \noindent where ${\cal H}_{r,\,A}^{p,\,q}(X)\subset C^\infty_{p,\,q}(X)$ is the space of $E_r$-Aeppli harmonic $(p,\,q)$-forms associated with $\omega$.

\end{Cor-Def}

Of course, the above isomorphism maps every class $\{\alpha\}_{E_{r,\,A}}\in E_{r,\,A}^{p,\,q}(X)$ to its unique $E_r$-Aeppli harmonic representative.

\vspace{3ex}

We can now conclude from the above results that there is a Serre-type {\bf canonical duality} between the $E_r$-Bott-Chern cohomology and the $E_r$-Aeppli cohomology of complementary bidegrees.

\begin{The}\label{The:E_r_BC-A_duality} Let $X$ be a compact complex manifold with $\mbox{dim}_\C X=n$. For all $p,q\in\{0,\dots , n\}$, the following bilinear pairing is {\bf well defined} and {\bf non-degenerate}: \[E_{r,\,BC}^{p,\,q}(X)\times E_{r,\,A}^{n-p,\,n-q}(X)\longrightarrow\C, \hspace{3ex} \bigg(\{\alpha\}_{E_{r,\,BC}},\,\{\beta\}_{E_{r,\,A}}\bigg)\mapsto\int\limits_X\alpha\wedge\beta.\]

\end{The}

\noindent {\it Proof.} The well-definedness was proved in Proposition \ref{Prop:E_r_BC-A_duality1}. The non-degeneracy is proved in the usual way on the back of the above preliminary results, as follows. 

Let $\{\alpha\}_{E_{r,\,BC}}\in E_{r,\,BC}^{p,\,q}(X)$ be an arbitrary non-zero class. Fix an arbitrary Hermitian metric $\omega$ on $X$ and let $\alpha$ be the unique {\it $E_r$-Bott-Chern harmonic} representative (w.r.t. $\omega$) of the class $\{\alpha\}_{E_{r,\,BC}}$ (whose existence and uniqueness are guaranteed by Corollary and Definition \ref{Cor-Def:E_r-BC_Hodge-isom}). In particular, $\alpha\neq 0$. 

Based on the characterisations of the $E_r$-Bott-Chern and $E_r$-Aeppli harmonicities given in Corollaries and Definitions \ref{Cor-Def:E_r-BC_Hodge-isom} and \ref{Cor-Def:E_r-A_Hodge-isom}, Lemma \ref{Lem:E_rE_r-bar_duality} and the standard equivalences ($\alpha\in\ker\partial\iff\star\bar\alpha\in\ker\partial^\star$) and ($\alpha\in\ker\bar\partial\iff\star\bar\alpha\in\ker\bar\partial^\star$) ensure that $\star\bar\alpha$ is {\it $E_r$-Aeppli harmonic}. In particular, $\star\bar\alpha$ represents an $E_r$-Aeppli class $\{\star\bar\alpha\}_{E_{r,\,A}}\in E_{r,\,A}^{n-p,\,n-q}(X)$. Moreover, pairing $\{\alpha\}_{E_{r,\,BC}}$ with $\{\star\bar\alpha\}_{E_{r,\,A}}$ yields $\int_X\alpha\wedge\star\bar\alpha = ||\alpha||^2\neq 0$, where $||\,\,||$ stands for the $L^2_\omega$-norm.   

Similarly, starting off with a non-zero class $\{\beta\}_{E_{r,\,A}}\in E_{r,\,A}^{n-p,\,n-q}(X)$ and selecting its unique {\it $E_r$-Aeppli harmonic} representative $\beta$, we get that $\beta\neq 0$, $\star\bar\beta$ is {\it $E_r$-Bott-Chern harmonic} (hence it represents a class in $E_{r,\,BC}^{p,\,q}(X)$) and the classes $\{\star\bar\beta\}_{E_{r,\,BC}}$ and $\{\beta\}_{E_{r,\,A}}$ pair to $\pm\,||\beta||^2\neq 0$.  \hfill $\Box$

\section{Characterisations of page-$r$-$\partial\bar\partial$-manifolds}\label{section:Characterisations in terms of exactness properties}

In this section, we apply the higher-page Bott-Chern and Aeppli cohomologies to give various characterisations of the page-$r$-$\partial\bar\partial$-manifolds introduced in [PSU20]. We will write $E_{r,\, BC}(X):=\bigoplus_{p,q} E_{r,\, BC}^{p,\, q}(X)$, $E_{r,\, BC}^k(X):=\bigoplus_{p+q=k} E_{r,\, BC}^{p,\, q}(X)$, $e_{r,\, BC}^k(X):=\dim E_{r,\, BC}^k(X)$ and similarly for $E_{r,\, A}$.

\begin{The}\label{The: page-r-Characterisations}
For a compact complex manifold $X$, the following properties are equivalent:
\begin{enumerate}
\item[(A)] $X$ is a \textbf{page-$(r-1)$-$\partial\bar\partial$-manifold}.
\item[(B)] The map $E_{r,\, BC}(X)\rightarrow E_{r,\, A}(X)$ induced by the identity is an {\bf isomorphism}.
\item[(C)] One has $e_{r,\, BC}^k(X)=e_{r,\, A}^k(X)$ for all $k$.
\item[(D)] The map $E_{r,\, BC}(X)\rightarrow E_{r,\, A}(X)$ induced by the identity is {\bf injective}.
\item[(E)] For any $d$-closed $(p,q)$-form $\alpha$, the following properties are equivalent:
\[
\alpha\text{ is } d\text{-exact}\Longleftrightarrow\alpha\text{ is } E_r\text{-exact}\Longleftrightarrow\alpha\text{ is } \bar E_r\text{-exact}\Longleftrightarrow\alpha\text{ is }E_r\bar E_r\text{-exact}.
\]
\end{enumerate}

Moreover, if the equivalent conditions (A)--(E) hold, the maps $E_{r,\,BC}^{p,\,q}(X)\longrightarrow E_r^{p,\,q}(X)$ and $E_{r}^{p,\,q}(X)\longrightarrow E_{r,\, A}^{p,\,q}(X)$ are {\bf isomorphisms} as well.

\end{The}

Let us first sketch how one might approach this theorem in an elementary way.
Fix $r\in\N_{>0}$ and a bidegree $(p,\,q)$. As in section \ref{subsection:duality_Er}, let ${\cal Z}_r^{p,\,q}$ and ${\cal C}_r^{p,\,q}$ stand for the space of $E_r$-closed, resp. $E_r$-exact, smooth $(p,\,q)$-forms on $X$. Let ${\cal D}_r^{p,\,q}$ stand for the space of $E_r\overline{E}_r$-exact smooth $(p,\,q)$-forms on $X$.

\begin{Lem}\label{Lem:inclusions_ddbar_r} (i)\, The following inclusions of vector subspaces of $C^\infty_{p+1,\,q}(X)$ hold: \begin{equation}\begin{array}{rcl}\label{eqn:inclusions_ddbar_r}\mbox{Im}\,(\partial\bar\partial)^{p,\, q}\subset\partial({\cal Z}_r^{p,\,q})\subset & {\cal D}_r^{p+1,\,q} & \subset{\cal C}_r^{p+1,\,q}\cap\ker d \\
	& \cap &   \\
	& \mbox{Im}\,d &\end{array}\end{equation}
	
	(ii)\, Every $E_r$-class $\{\alpha\}_{E_r}\in E_r^{p,\,q}(X)$ can be represented by a $d$-closed form if and only if $\partial({\cal Z}_r^{p,\,q})\subset\mbox{Im}\,(\partial\bar\partial)$. In other words, this happens if and only if the first inclusion in (\ref{eqn:inclusions_ddbar_r}) is an equality.
	
\end{Lem}

\noindent {\it Proof.} (i)\, To prove the first inclusion, it suffices to show that every $\bar\partial$-exact $(p,\,q)$-form is $E_r$-closed. Let $\alpha=\bar\partial\beta$ be a $(p,\,q)$-form. Then, $\bar\partial\alpha=0$ and $\partial\alpha = \bar\partial(-\partial\beta)$. Putting $u_1:=-\partial\beta$, we have $\partial u_1=0$, so we can choose $u_2=0,\dots , u_{r-1}=0$ to satisfy the tower of equations under (i) of Proposition \ref{Prop:E_r-closed-exact_conditions}. This shows that $\alpha$ is $E_r$-closed.

To prove the second inclusion, let $\alpha\in{\cal Z}_r^{p,\,q}$. By (i) of Proposition \ref{Prop:E_r-closed-exact_conditions}, this implies that $\bar\partial\alpha=0$, so if we write $\partial\alpha = \partial\zeta + \partial\bar\partial\xi + \bar\partial\eta$ with $\zeta=\alpha$, $\xi=0$ and $\eta=0$, we satisfy the conditions under (ii) of Definition \ref{Def:E_rE_r-bar} with $v_j=0$ and $u_j=0$ for all $j\in\{0,\dots , r-3\}$. This proves that $\partial\alpha$ is $E_r\overline{E}_r$-exact.

The third inclusion on the first row is a consequence of (iii) and (iv) of Lemma \ref{Lem:E_rE_r-bar-properties}, while the ``vertical'' inclusion is a translation of (iv) of the same lemma.

\vspace{2ex}

(ii)\, Let $\{\alpha\}_{E_r}\in E_r^{p,\,q}(X)$ be an arbitrary class and let $\alpha$ be an arbitrary representative of it. Then, $\{\alpha\}_{E_r}\in E_r^{p,\,q}(X)$ can be represented by a $d$-closed form if and only if there exists an $E_r$-exact form $\rho=\partial a + \bar\partial b$, with $a$ satisfying the conditions $\bar\partial a = \partial c_{r-3}$, $\bar\partial c_{r-3} = \partial c_{r-4},\dots \bar\partial c_0=0$ for some forms $c_j$, such that $\partial(\alpha-\rho)=0$. This last identity is equivalent to \[\partial\bar\partial b = \partial\alpha.\]

\noindent Thus, the class $\{\alpha\}_{E_r}$ contains a $d$-closed form if and only if the form $\partial\alpha$, which already lies in $\partial({\cal Z}_r^{p,\,q})$, is $\partial\bar\partial$-exact. This proves the contention.  \hfill $\Box$

\begin{The}\label{The:page_r-1_ddbar-equivalence_bis} Let $X$ be a compact complex manifold with $\mbox{dim}_\C X=n$. Fix an arbitrary integer $r\geq 2$. The following properties are equivalent.
	
	\vspace{1ex}
	
	(A)\, $X$ is a {\bf page-$(r-1)$-$\partial\bar\partial$-manifold}.
	
	\vspace{1ex}
	
	(F)\, For all $p,q\in\{0,\dots , n\}$, the following identities of vector subspaces of $C^\infty_{p+1,\,q}(X)$ hold:
	
	\begin{equation*}\label{eqn:identities_ddbar_r_bis}(i)\hspace{2ex} \mbox{Im}\,(\partial\bar\partial) = \partial({\cal Z}_r^{p,\,q})  \hspace{3ex} \mbox{and} \hspace{3ex} (ii)\hspace{2ex} {\cal C}_r^{p,\,q}\cap\ker d = \mbox{Im}\,d.\end{equation*}
	
\end{The}

\noindent {\it Proof.} By (ii) of Lemma \ref{Lem:inclusions_ddbar_r}, identity (i) in $(F)$ is equivalent to every $E_r$-class of type $(p,\,q)$ being representable by a $d$-closed form. On the other hand, if this is the case, then identity (ii) in $(F)$ is equivalent to the map $E_r^{p,\,q}(X)\ni\{\alpha\}_{E_r}\mapsto\{\alpha\}_{DR}\in H^{p+q}_{DR}(X,\,\C)$ (with $\alpha\in\ker d$) being well defined and injective, which means $X$ is page-$(r-1)$-$\partial\bar\partial$.  \hfill $\Box$

\vspace{3ex}

Note that identity (ii) in $(F)$ of Theorem \ref{The:page_r-1_ddbar-equivalence_bis} is a reformulation of the first equivalence in $(E)$ of Theorem \ref{The: page-r-Characterisations}. However, it appears to be nontrivial to obtain the whole statement of Theorem \ref{The: page-r-Characterisations} in an elementary way, so we will now make use of some algebro-cohomological machinery. As a preparation, let us compute the higher-page Bott-Chern and Aeppli cohomologies for all indecomposable double complexes. Recall that a double complex is a bigraded vector space $A$ with maps $\partial_1,\partial_2$ of bidegree $(1,0)$ and $(0,1)$ s.t. $d:=\partial_1+\partial_2$ satisfies $d^2=0$. It is bounded if $A^{p,q}=0$ for all but finitely many $(p,q)\in\Z^2$ and indecomposable if it cannot be written as a nontrivial direct sum $A=B\oplus C$. There are two kinds of indecomposable double complexes (see [KQ20], [Ste20]).\\

\noindent{\bf Squares}\\
These are double complexes generated by an element $a$ of pure bidegree $(p,q)$ s.t. $\partial_1 \partial_2 a\neq 0$:
\[
\begin{tikzcd}
\langle\partial_2 a\rangle\ar[r]&\langle\partial_2\partial_1  a\rangle\\
\langle a\rangle\ar[u]\ar[r]&\langle\partial_1 a\rangle\ar[u].
\end{tikzcd}
\]

\noindent{\bf Zigzags}\\
A zigzag is a double complex generated by pure-type elements $a_1,...,a_l$ some fixed total degree $k=p+q$ and satisfying $\del_1\del_2 a_i=0$ for all $i$ and $\partial_1 a_i=\partial_2 a_{i+1}\neq 0$ for all $i$ s.t. $a_i\neq 0\neq a_{i+1}$. In particular, $\partial_2 a_1$ and $\partial_1a_l$ may or may not be zero. 
\[
 \begin{tikzcd}
 &\langle \partial_2 a_1\rangle&&&\\
  & \langle a_1\rangle \arrow{r}\arrow{u} & \langle \partial_1 a_1 \rangle \\
  & & \langle a_2 \rangle \arrow{u} \arrow{r} & \arrow[draw=none]{rd}{\cdots} \\
  & & & & { } \\
  & & & & \langle a_l \rangle \arrow{u} \arrow{r} & \langle \partial_1 a_l \rangle.
 \end{tikzcd}
\]
The length of a zigzag is given by its dimension as a vector space. It has to be $2l-1,2l$ or $2l+1$, depending on wether the outmost differentials vanish or not. A zigzag of length one is called a dot.\\

In order to compute $E_{r,\, BC}$ and $E_{r,\, A}$ on those double complexes, first recall that $E_{r,\, BC}$ is a quotient of $H_{BC}$ and $E_{r,\, A}$ is a subspace of $H_A$. In particular, if $H_{BC}$ or $H_{A}$ are zero on some double complex, so are their lower dimensional counterparts. Further note that $\partial_1\partial_2=0$ on zigzags and we even have $\partial_1=\partial_2=0$ on dots. These observations yield

\begin{Obs}
For a square $S$ we have $E_{r,\, A}(S)=E_{r,\, BC}(S)=0$, while for a dot $D=\langle a\rangle$ we have $E_{r,\, BC}(D)=E_{r,\, A}(D)=\langle a\rangle$ for all $r\geq 1$.
\end{Obs}
%
%
For higher length zigzags $Z$, generated by $a_1,...,a_l$, we get that $H_{A}(Z)=\langle a_1,...,a_l\rangle$ keeps the lower antidiagonal, while $H_{BC}(Z)=\langle \partial_2 a_1,....,\partial_2 a_l, \partial_1 a_l\rangle$ remembers the higher antidiagonal. To describe their higher-page analogues, it suffices to understand the kernel, resp. cokernel, of the projection $H_{BC}(Z)\twoheadrightarrow E_{r,\, BC}(Z)$, resp. the inclusion $E_{r,\, A}(Z)\hookrightarrow H_{A}(Z)$. These are described as follows.

\begin{Lem}\label{Lem: BCA on indecomposables}
Let $Z$ be a zigzag of length at least two, generated by $a_1,...,a_l$. For any $i\not\in\{1,...,l\}$, set $a_i:=0$. Then, for any $r\geq 2$, the contributions of different zigzags are as follows:

\begin{enumerate}
\item {\bf Even length type $I$}: If $\partial_2 a_1=0$ and $\partial_1 a_l\neq 0$, one has:
\begin{align*}
\ker\left(H_{BC}(Z)\rightarrow E_{r,\, BC}(Z)\right)&=\langle \partial_1 a_1,...,\partial_1 a_{r-1}\rangle\\
\operatorname{coker}\left(E_{r,\, A}(Z)\rightarrow H_A(Z)\right)&=\langle a_{l-r+2},...,a_l\rangle.
\end{align*}
\item {\bf Even length type $II$}: If $\partial_2 a_1\neq 0$ and $\partial_1 a_l= 0$, one has:
\begin{align*}
\ker\left(H_{BC}(Z)\rightarrow E_{r,\, BC}(Z)\right)&=\langle \partial_1 a_{l-r+2},...,\partial_1 a_{l}\rangle\\
\operatorname{coker}\left(E_{r,\, A}(Z)\rightarrow H_A(Z)\right)&=\langle a_{1},...,a_{r-1}\rangle.
\end{align*}
\item  {\bf Odd length type $M$}: If $\partial_2a_1=0=\partial_2 a_l$, one has $E_{r,\, A}(Z)=H_A(Z)$ and
\[
\ker\left(H_{BC}(Z)\rightarrow E_{r,\, BC}(Z)\right)=\langle\partial_1 a_1,...,\partial_1 a_{r-1},\partial 2 a_{l-r+2},...,\partial_2 a_l\rangle.
\] 
\item {\bf Odd length type $L$}: If $\partial_2 a_1\neq 0\neq \partial_2 a_l$, one has $H_{BC}(Z)=E_{r,\, BC}(Z)$ and
\[
\operatorname{coker}\left(E_{r,\, A}(Z)\rightarrow H_A(Z)\right)=\langle a_1,...,a_{r-1}, a_{l-r+2},...,a_l\rangle.
\]
\end{enumerate}
Note that for large $r$, some of the written generators could be zero or there could be some overlap in the last two cases.
\end{Lem}
\begin{proof}
Let us only do the computation for $E_{r,\, BC}$. The elements that get modded out here in addition to the $\partial_1\partial_2$-exact ones are the $E_r$-exact ones and the $\bar E_r$-exact ones. 
By the definition of $ E_r$-exactness, this means that, whenever a zigzag has top left corner generated by $a_1$ with $\partial_2 a_1=0$, i.e.:
\[
\begin{tikzcd}
\langle a_1\rangle\ar[r]&\langle\partial_1 a_1\rangle&&&\\
&\langle a_2\rangle\ar[u]\ar[r]&\langle\partial_1 a_2\rangle&&\\
&&&\ddots&
\end{tikzcd}
\]
the classes of $\partial_1 a_1,...,\partial_1 a_{r-1}$ are zero in $E_{r,\, BC}(Z)$. Along the same lines, if a zigzag has bottom right corner generated by $a_1$ with $\partial_2 a_1=0$, i.e.:
\[
\begin{tikzcd}
\ddots&{}&&&\\
&\langle a_{2}\rangle\ar[u]\ar[r]&\langle\partial_2 a_1\rangle &&\\
&&\langle a_1\rangle\ar[u]&&
\end{tikzcd}
\]
the classes of $\partial_2 a_1,...,\partial_2 a_{r-1}$ are zero in $E_{r,\, BC}(Z)$. This yields the result for $E_{r,\, BC}$. The calculation for $H_A$ is analogous. 
\end{proof}
Let us also record what this yields for the dimensions of the new cohomology groups.
\begin{Cor}\label{Cor: dimensions of BCA on indecomposables}Let $Z$ be an indecomposable bounded double complex and let $r\geq 2$.
\begin{enumerate}
\item If $Z$ is a square, then $e_{r,\, BC}(Z)=0=e_{r,\, A}(Z)$.
\item If $Z$ is a dot, then $e_{r,\, BC}(Z)=1=e_{r,\, A}(Z)$.
\item If $Z$ is a zigzag of odd length $2l+1\geq 3$ of type $L$, one has $e_{r,\, BC}(Z)=h_{BC}(Z)=l+1$ and $e_{r,\, A}(Z)=\max\{l-2(r-1),0\}$.
\item If $Z$ is a zigzag of odd length $2l+1\geq 3$ of type $M$, one has $e_{r,\, BC}(Z)=\max\{l-2(r-1), 0\}$ and $e_{r,\, A}(Z)=h_{A}(Z)=l+1$.
\item If $Z$ is a zigzag of even length $2l$, one has $e_{r,\, BC}(Z)=e_{r,\, A}(Z)=\max\{l-r+1,0\}$.
\end{enumerate}
\end{Cor}
Recall ([KQ20], [PSU20], [Ste20]) that any bounded double complex $A$ can be written as a direct sum of indecomposable ones, all indecomposable ones are either squares or zigzags and $A$ has the page-$r$-$\partial_1\partial_2$-property if and only if in any decomposition into indecomposables, there are no odd length zigzags other than dots (length one) and no even length zigzags of length greater than $2r$. Thus, we get
\begin{Cor}
For any bounded double complex $A$ such that all the numerical quantities involved are finite, there is an inequality:
\[
e_{r,\, A}(A)+e_{r,\, BC}(A)\geq e_r(A)+\bar{e}_r(A)\geq 2b(A).
\]
Equality holds if $A$ satisfies the page-$(r-1)$-$\partial_1\partial_2$-property
\end{Cor}

\begin{proof}
Since all the quantities involved are additive under direct sums, it suffices to show this for indecomposable double complexes $Z$. The middle and right hand side were computed in [Ste20]: $e_r(Z)+e_r(Z)$ equals $0$ for a square and for a zigzag of length $2l\leq 2(r-1)$, while it equals $2$ for all other zigzags. Also, $b(Z)=1$ for odd length zigzags and $b(Z)=0$ otherwise. In particular, for an arbitrary double complex, the middle quantity is just twice the number of all zigzags which have odd length or even length at least $2r$. By the previous Corollary \ref{Cor: dimensions of BCA on indecomposables}, $e_{r,\, A}(Z)+e_{r,\, BC}(Z)=0$ for squares and even length zigzags of length $2l\leq 2(r-1)$, while it equals $2$ for dots and zigags of length $2r$ and is greater than or equal to $2$ for all other zigzags.
\end{proof}
\begin{Rem}
Somewhat unexpectedly, the equality $e_{r,\, A}(A)+e_{r,\, BC}(A)=2b(A)$ does \textbf{not} imply the page-$(r-1)$-$\partial_1\partial_2$-property for $r\geq 2$, contrary to the case $r=1$ ([AT13]). For example, both sides are equal to $2$ for $r\geq 2$ and $A$ a zigzag of length $3$. As one may see, for example from a Hopf surface, this behaviour really occurs in geometric situations. A different generalisation of the case $r=1$ has been obtained in [PSU20].
\end{Rem}

\begin{Cor}Let $A$ be a bounded double complex such that $e_{r,\, BC}^k(A)$ and $e_{r,\, A}^k(A)$ are finite. The following properties are equivalent:
\begin{enumerate}
\item[(A')] $A$ has the \textbf{page-$(r-1)$-$\partial_1\partial_2$-property}.
\item[(B')] The map $E_{r,\, BC}(A)\rightarrow E_{r,\, A}(A)$ is an isomorphism.
\item[(C')] One has $e_{r,\, BC}^k(A)=e_{r,\, A}^k(A)$ for all $k\in\Z$.
\end{enumerate}
\end{Cor} 

\begin{proof}
Note that property {\it(B')} implies property {\it(C')}. Thus, it suffices to show that property {\it(B')} is  satisfied for squares, dots and even length zigzags of length $\leq 2(r-1)$ and property {\it(C')} is violated for all other zigzags. This is a direct consequence of Lemma \ref{Lem: BCA on indecomposables} and Corollary \ref{Cor: dimensions of BCA on indecomposables}.
\end{proof}

\begin{proof}[Proof of Theorem \ref{The: page-r-Characterisations}] 
Only the last two points remain to be proved. By the duality between $E_{r,\, BC}(X)$ and $E_{r,\, A}(X)$, both spaces have the same dimension, so $(D)\Leftrightarrow (B)$.\\
Property $(E)$ can be restated as the fact that all the maps originating from $E_{r,\, BC}^{p,\, q}(X)$ (including the one with target in $E_{r,\, A}^{p,\,q}(A)$) in the following commutative diagram are injective:
\[
\begin{tikzcd}
&E_{r,\, BC}^{p,\, q}(X)\ar[ld]\ar[d]\ar[rd]&\\
E_r^{p,\, q}(X)\ar[rd]&H_{dR}^{p+q}(X,\,\C)\ar[d]&\bar E_r^{p,\, q}(X)\ar[ld]\\
&E_{r,\, A}^{p,\, q}(X)&
\end{tikzcd}
\]
In particular, $(E)\Rightarrow (D)$. Since the composition of two maps being injective implies the right-hand map is injective, we get $(D)\Rightarrow(E)$.

\vspace{1ex}

The last statement of Theorem \ref{The: page-r-Characterisations} follows from the above arguments. \end{proof}

\vspace{2ex}

The equivalence between $(A)$ and $(D)$ in Theorem \ref{The: page-r-Characterisations} is an extension to $r>1$ of Theorem 3.1. in [AT17]. Together with Proposition \ref{Prop:E_r_BC-A_duality1}, Theorem \ref{The: page-r-Characterisations} also implies the following extension of Theorem 5.2. in [AT17].

\begin{Cor}
For a compact complex manifold $X$ of dimension $n$, the canonical bilinear pairing in $E_r$-Bott-Chern cohomology $E^{p,\, q}_{r,\, BC}(X) \times E^{n-p,\, n-q}_{r,\, BC}(X) \longrightarrow \C$, given by wedge product and integration, is well defined. Moreover, $X$ is a page-$(r-1)$-$\partial\bar\partial$-manifold if and only if the pairing is non-degenerate.
\end{Cor}
\section{Relations with the $E_r$-sG manifolds}\label{section:relations_Er-sG}

Let $X$ be an $n$-dimensional compact complex manifold and let $\omega$ be a Gauduchon metric on $X$. This means that $\omega$ is a smooth, positive definite $(1,\,1)$-form on $X$ such that $\partial\bar\partial\omega^{n-1}=0$. Gauduchon metrics were introduced and proved to always exist in [Gau77]. It was noticed in [Pop19] that $\partial\omega^{n-1}$ is $E_r$-closed for every $r\in\N_{>0}$ and the following definition was introduced.

\begin{Def}([Pop19])\label{Def:E_r_sG} Let $r\in\N_{>0}$. A Gauduchon metric $\omega$ on $X$ is said to be {\bf $E_r$-sG} if $\partial\omega^{n-1}$ is {\bf $E_r$-exact}.

  We say that $X$ is an {\bf $E_r$-sG manifold} if an $E_r$-sG metric $\omega$ exists on $X$.

\end{Def}

The $E_1$-sG property coincides with the {\it strongly Gauduchon (sG)} property introduced in [Pop13]. The following implications and their analogues for $X$ are obvious:

\vspace{1ex}

\hspace{12ex} $\omega$ is $E_1$-sG $\implies$ $\omega$ is $E_2$-sG $\implies$ $\omega$ is $E_3$-sG

\vspace{1ex}

\noindent and so is the fact that, for bidegree reasons, the $E_r$-sG property coincides with the $E_3$-sG property for all $r\geq 4$.

The link between the page-$(r-1)$-$\partial\bar\partial$ and the $E_r$-sG properties is spelt out in the following

\begin{Prop}\label{Prop:page_r-1_Er-sG} Let $r\in\N_{>0}$ and let $X$ be a page-$(r-1)$-$\partial\bar\partial$-manifold. Then, every Gauduchon metric on $X$ is $E_r$-sG. In particular, $X$ is an $E_r$-sG manifold.

\end{Prop}

\noindent {\it Proof.} Let $\omega$ be a Gauduchon metric on $X$. Then, $\partial\omega^{n-1}$ is $\bar\partial$-closed and $\partial$-closed, hence $d$-closed. It is also $\partial$-exact, or equivalently, $\overline{E}_1$-exact, hence also $\overline{E}_r$-exact.

Now, thanks to Theorem \ref{The: page-r-Characterisations}, the page-$(r-1)$-$\partial\bar\partial$-property of $X$ implies the equivalence between $\overline{E}_r$-exactness and $E_r$-exactness for $d$-closed pure-type forms. Consequently, $\partial\omega^{n-1}$ must be $E_r$-exact, so $\omega$ is an $E_r$-sG metric.  \hfill $\Box$

\vspace{3ex}

Let $X_{u,v}$ be a Calabi-Eckmann manifold, i.e. any of the complex manifolds $C^\infty$-diffeomorphic to $S^{2u+1}\times S^{2v+1}$ constructed by Calabi and Eckmann in [CE53]. Recall that the $X_{0,v}$'s and the $X_{u,0}$'s are Hopf manifolds. By [Pop14], $X_{u,v}$ does not admit any sG metric. However, we now prove the existence of $E_2$-sG metrics when $uv>0$.

\begin{Prop}\label{Calabi-Eckmann-prop}
Let $X_{u,v}$ be a Calabi-Eckmann manifold 
of complex dimension $\geq 2$. Let $u\leq v$. 
\begin{enumerate}
\item[\rm (i)] If $u>0$, then $X_{u,v}$ does not admit $sG$ metrics, 
but it is an $E_2$-sG manifold.
\item[\rm (ii)] If $u=0$, $X_{u,v}$ does not admit $E_r$-sG metrics for any $r$.
\end{enumerate}
\end{Prop}

\begin{proof}
By Borel's result in [Hir78, Appendix Two by A. Borel], we have
$$
H^{\bullet,\bullet}_{\bar\partial}(X_{u,v}) \cong \frac{{\mathbb C} [x_{1,1}]}{(x_{1,1}^{u+1})} \otimes \bigwedge (x_{v+1,v},x_{0,1}).
$$
In other words, a model for the Dolbeault cohomology of the Calabi-Eckmann manifold $X_{u,v}$ is provided by the CDGA (see [NT78])
$$
(V\langle x_{0,1},\, x_{1,1},\, y_{u+1,u},\, x_{v+1,v} \rangle, \ \bar\partial),
$$
with differential 
$$
\bar\partial x_{0,1}=0,\quad \bar\partial x_{1,1}=0,\quad 
\bar\partial y_{u+1,u}=x_{1,1}^{u+1}, \quad \bar\partial x_{v+1,v}=0.
$$
Thus, if $u>0$, we have a minimal model. 
(For $u=0$ we have only a cofibrant model in the sense of [NT78].)

Moreover, $\partial$ acts on generators as follows [NT78]:
$$
\partial x_{0,1}=x_{1,1}\ \  (\mbox{hence } \partial x_{1,1}=0),\quad 
\partial y_{u+1,u}=0, \quad \partial x_{v+1,v}=0.
$$
Next we determine the spaces $E_{r}^{n,\, n-1}$, for any $r\geq 1$,
where $n=u+v+1$. 

Let us first focus on the case (i), i.e. $u>0$. 
We need to consider the Dolbeault cohomology groups $H^{n-1,\,n-1}_{\bar\partial}(X_{u,v})$ and $H^{n,\,n-1}_{\bar\partial}(X_{u,v})$. They are given by
$$
H^{u+v,\,u+v}_{\bar\partial}(X_{u,v}) = \langle x_{0,1}\!\cdot x_{1,1}^{u-1}\!\cdot  x_{v+1,v} \rangle, \quad
H^{u+v+1,\,u+v}_{\bar\partial}(X_{u,v}) = \langle x_{1,1}^{u}\!\cdot  x_{v+1,v} \rangle.
$$
Now we consider
$$
H^{u+v,\,u+v}_{\bar\partial}(X_{u,v}) \stackrel{\partial}{\longrightarrow} H^{u+v+1,\,u+v}_{\bar\partial}(X_{u,v}) 
\longrightarrow 0.
$$
Since $\partial(x_{0,1}\!\cdot x_{1,1}^{u-1}\!\cdot  x_{v+1,v}) = 
\partial(x_{0,1})\!\cdot x_{1,1}^{u-1}\!\cdot  x_{v+1,v} 
= x_{1,1}^{u}\!\cdot  x_{v+1,v}$, the first map is surjective. 
Therefore, $E_2^{n,\,n-1}(X_{u,v})=0$. 
Thus, any Gauduchon metric on $X_{u,v}$ is an $E_2$-sG metric.

Next we focus on the case (ii), i.e. $u=0$, so $v\geq 1$. 
In this case $H^{v+1,v}_{\bar\partial}(X_{0,v})=\langle x_{v+1,v} \rangle$. Notice that the Dolbeault cohomology groups
$H^{v-r+1,v+r-1}_{\bar\partial}(X_{0,v})$ are all zero for every $r\geq 1$. Therefore,
$E^{v-r+1,v+r-1}_r(X_{0,v})=\{0\}$ for every $r\geq 1$. Meanwhile, from 
$$
\{0\}=E^{v-r+1,v+r-1}_r(X_{0,v}) \stackrel{d_r}{\longrightarrow} E^{v+1,v}_r(X_{0,v}) 
\longrightarrow 0,
$$
we get $E^{v+1,v}_r(X_{0,v})=H^{v+1,v}_{\bar\partial}(X_{0,v})$ for every $r\geq 2$. 
So, the existence of an $E_r$-sG metric on $X_{0,v}$ would imply the existence of an sG metric, which would contradict [Pop14].

\end{proof}


As a by-product of Borel's description of the Dolbeault cohomology of the Calabi-Eckmann manifolds $X_{u,v}$ used in the above proof, one gets that the Fr\"olicher spectral sequence of $X_{u,v}$ satisfies $E_1\ne E_2= E_{\infty}$ when $u>0$, whereas it degenerates at $E_1$ when $u=0$. This latter fact implies that no Hopf manifold $X_{0,v}$ can have a pure De Rham cohomology, hence cannot be a page-$r$-$\partial\bar\partial$-manifold for any $r$. Indeed, if the De Rham cohomology were pure, then $X_{0,v}$ would be a $\partial\bar\partial$-manifold, a fact that is trivial to contradict.

By the previous result, all the Calabi-Eckmann manifolds that are not Hopf manifolds are $E_2$-sG manifolds. However, the next observation shows that they are not page-$r$-$\partial\bar\partial$-manifolds for any $r\in\N$.

\begin{Lem}\label{Lem:C-E_not-DR-pure} Let $u,v\geq 0$ and let $X=S^{2u+1}\times S^{2v+1}$ be equipped with any of the Calabi-Eckmann complex structures. Assume that either $u\neq v$ or $u=v=1$. Then, the De Rham cohomology of $X$ is not pure.
\end{Lem}  

\begin{proof}
It is a well-known consequence of the result in [Hir78, Appendix Two by A. Borel] recalled above that, for any Calabi-Eckmann manifold, the Fr\"olicher spectral sequence degenerates at $E_2$. Thus, if the De Rham cohomology were pure, $X$ would be page-$r$-$\partial\bar\partial$. By [PSU20, Thm 3.1], this would imply $h_{BC}(X)=h_{\bar\partial}(X)$, where $h_{BC}(X):=\sum_{p,\, q}h^{p,\, q}_{BC}(X)$ and $h_{\bar\partial}(X):=\sum_{p,\, q}h^{p,\, q}_{\bar\partial}(X)$ are the total Bott-Chern and Hodge numbers. For $u\neq v$, the Bott-Chern numbers were computed in [Ste20, Cor. H], yielding $h_{BC}(X)=4u+5>4u+4=h_{\delbar}(X)$. A calculation in [TT17] implies that this continues to hold for $u=v=1$.

\end{proof}
\begin{Rem}

  This lemma is very likely to also hold for arbitrary $u=v>1$. From the proof of [Ste20, Cor H] it follows that to settle this issue, it suffices to determine whether $h^{v+1,u+1}_{BC}=1$ for the higher-dimensional Calabi-Eckmann manifolds with $u=v$. Since we are only interested in the existence of a counterexample here, we do not pursue this issue.
\end{Rem}

\section{Further applications to H-S, SKT and sGG manifolds}\label{section:applications_HS-SKT-sGG}

We start by recalling the following definitions (see [ST10] for the notion of {\it Hermitian-symplectic (H-S) metric and manifold}; [Pop13] for the notion of {\it strongly Gauduchon (sG) metric and manifold}; [PU18] for the notion of {\it sGG manifold}).

\begin{Def}\label{Def:HS-SKT-sGG} Let $X$ be a compact complex manifold with $\mbox{dim}_\C X=n$.

  (a)\, A Hermitian metric $\omega$ on $X$ is said to be

  \vspace{1ex}

  (i)\, {\bf SKT} if $\partial\bar\partial\omega=0$;

  \vspace{1ex}

  (ii)\, {\bf Hermitian-symplectic (H-S)} if there exists a form $\rho^{0,\,2}\in C^\infty_{0,\,2}(X,\,\C)$ such that $$d(\overline{\rho^{0,\,2}} + \omega + \rho^{0,\,2})=0;$$

\vspace{1ex}

  (iii)\, {\bf strongly Gauduchon (sG)} if $\partial\omega^{n-1}$ is $\bar\partial$-exact.

\vspace{1ex}

(b)\, If $X$ admits an SKT, or an H-S, or an sG metric $\omega$, then $X$ is said to be an {\bf SKT manifold}, resp. an {\bf H-S manifold}, resp. an {\bf sG manifold}.

\vspace{1ex}

(c)\, If every Gauduchon metric on $X$ is strongly Gauduchon (sG), then $X$ is said to be an {\bf sGG manifold}.

\end{Def}

\vspace{2ex}

The notion of sG metric is the analogue in bidegree $(n-1,\,n-1)$ of the notion of H-S metric. Indeed, by [Pop13, Proposition 4.2], a metric $\omega$ is sG if and only if there exists a form $\Omega^{n-2,\,n}\in C^\infty_{n-2,\,n}(X,\,\C)$ such that $$d(\overline{\Omega^{n-2,\,n}} + \omega^{n-1} + \Omega^{n-2,\,n})=0.$$

\vspace{1ex}

Our first observation is that the {\it higher-page Aeppli cohomologies} introduced in this paper provide the natural cohomological framework for the study of the above metric notions. 

\begin{Prop}\label{Prop:sG-HS_E_2A} Let $X$ be a compact complex manifold with $\mbox{dim}_\C X=n$ and let $\omega$ be a Hermitian metric on $X$.

  \vspace{1ex}

  (i)\, The metric $\omega$ is {\bf strongly Gauduchon (sG)} if and only if $\omega^{n-1}$ is {\bf $E_2\overline{E}_2$-closed}.

  In particular, in this case, $\omega^{n-1}$ induces an $E_2$-Aeppli cohomology class $\{\omega^{n-1}\}_{E_{2,A}}\in E_{2,A}^{n-1,\,n-1}(X)$.

  \vspace{1ex}

  (ii)\, The metric $\omega$ is {\bf Hermitian-symplectic (H-S)} if and only if $\omega$ is {\bf $E_3\overline{E}_3$-closed}.

  In particular, in this case, $\omega$ induces an $E_3$-Aeppli cohomology class $\{\omega\}_{E_{3,A}}\in E_{3,A}^{1,\,1}(X)$.

  \vspace{1ex}

  When $n=3$, $\omega$ is {\bf Hermitian-symplectic (H-S)} if and only if $\omega$ is {\bf $E_2\overline{E}_2$-closed}.

In particular, in this case, $\omega$ induces an $E_2$-Aeppli cohomology class $\{\omega\}_{E_{2,A}}\in E_{2,A}^{1,\,1}(X)$.

\end{Prop}

\noindent {\it Proof.} (i)\, The sG condition on $\omega$ is defined by requiring $\partial\omega^{n-1}$ to be $\bar\partial$-exact. By conjugation, this is equivalent to $\bar\partial\omega^{n-1}$ being $\partial$-exact. These two conditions express the fact that $\omega$ satisfies the two towers of $1$ equation each (hence corresponding to the case $r=2$) in (i) of Definition \ref{Def:E_rE_r-bar}, so they are equivalent to $\omega^{n-1}$ being {\bf $E_2\overline{E}_2$-closed}.

(ii)\, The H-S condition on $\omega$ is equivalent to the existence of a form $\rho^{2,\,0}\in C^\infty_{2,\,0}(X,\,\C)$ such that $\partial\omega = -\bar\partial\rho^{2,\,0}$ and $\partial\rho^{2,\,0}=0$. By conjugation, these conditions are equivalent to the existence of a form $\rho^{0,\,2}\in C^\infty_{0,\,2}(X,\,\C)$ such that $\bar\partial\omega = -\partial\rho^{0,\,2}$ and $\bar\partial\rho^{0,\,2}=0$. These four conditions express the fact that $\omega$ satisfies the two towers of $2$ equations each (hence corresponding to the case $r=3$) in (i) of Definition \ref{Def:E_rE_r-bar}, or equivalently that $\omega$ is $E_3\overline{E}_3$-closed.

When $n=3$, the condition $\partial\rho^{2,\,0}=0$ is automatic since $$i\partial\rho^{2,\,0}\wedge\bar\partial\rho^{0,\,2} = \partial\rho^{2,\,0}\wedge\star\bar\partial\rho^{0,\,2} = |\partial\rho^{2,\,0}|^2_\omega\,dV_\omega\geq 0 \hspace{3ex} \mbox{and} \hspace{3ex} \int\limits_Xi\partial\rho^{2,\,0}\wedge\bar\partial\rho^{0,\,2} = \int\limits_X\partial(i\rho^{2,\,0}\wedge\bar\partial\rho^{0,\,2}) =0,$$ where $\star = \star_\omega$ is the Hodge star operator induced by $\omega$ and the primitivity (trivial for bidegree reasons) of the $(0,\,3)$-form $\bar\partial\rho^{0,\,2}$ was used to infer that $\star\bar\partial\rho^{0,\,2} = i\bar\partial\rho^{0,\,2}$. In fact, the last identity follows from the general formula: \begin{eqnarray*}\label{eqn:prim-form-star-formula-gen}\star\, v = (-1)^{k(k+1)/2}\, i^{p-q}\, \frac{\omega^{n-p-q}\wedge v}{(n-p-q)!}, \hspace{2ex} \mbox{where}\,\, k:=p+q,\end{eqnarray*} satisfied by any {\it primitive} form $v$ of any bidegree $(p, \, q)$ (see e.g. [Voi02, Proposition 6.29, p. 150]).   \hfill $\Box$

\vspace{3ex}

Using the real structure of the spaces $E_{r,\, A}^{1,\,1}(X)$, we will now consider their {\it real} versions $E_{r,\, A}^{1,\,1}(X,\,\R)$ and define {\it cones} of higher Aeppli cohomology classes representable by the kinds of metrics discussed above.

\begin{Def}\label{Def:E_rA_pos-cones} Let $X$ be a compact complex manifold with $\mbox{dim}_\C X=n$. We define the:

  \vspace{1ex}

 (i)\, {\bf strongly Gauduchon (sG) cone} of $X$ as the set: $${\cal SG}_X:=\bigg\{\{\omega^{n-1}\}_{E_{2,A}}\in E_{2,A}^{n-1,\,n-1}(X,\,\R)\,\mid\,\omega \hspace{1ex}\mbox{is an sG metric on}\hspace{1ex} X\bigg\}\subset E_{2,A}^{n-1,\,n-1}(X,\,\R);$$

 (ii)\, {\bf Hermitian-symplectic (H-S) cone} of $X$ as the set: $${\cal HS}_X:=\bigg\{\{\omega\}_{E_{3,A}}\in E_{3,A}^{1,\,1}(X,\,\R)\,\mid\,\omega \hspace{1ex}\mbox{is an H-S metric on}\hspace{1ex} X\bigg\}\subset E_{3,A}^{1,\,1}(X,\,\R);$$

 (iii)\, {\bf SKT cone} of $X$ as the set: $${\cal SKT}_X:=\bigg\{\{\omega\}_{E_{1,A}}\in E_{1,A}^{1,\,1}(X,\,\R)\,\mid\,\omega \hspace{1ex}\mbox{is an SKT metric on}\hspace{1ex} X\bigg\}\subset E_{1,A}^{1,\,1}(X,\,\R) = H_A^{1,\,1}(X,\,\R).$$

\end{Def}

The strongly Gauduchon cone was defined in a different way in [Pop15] and in [PU18], although it was denoted in the same way. We keep the same notation for the object defined above since it will be seen in Corollary \ref{Cor:cone-equalities} to play an identical role in characterising sGG manifolds. Meanwhile, recall that the {\it Gauduchon cone} ${\cal G}_X\subset E_{1,A}^{n-1,\,n-1}(X,\,\R) = H_A^{n-1,\,n-1}(X,\,\R)$ of $X$ was introduced in an analogous way in [Pop15] (i.e. as the set of the Aeppli cohomology classes of all the $\omega^{n-1}$ induced by Gauduchon metrics $\omega$) and shown there to be an {\it open convex cone} in $H_A^{n-1,\,n-1}(X,\,\R)$. The same conclusion holds for the new cones introduced above.

\begin{Lem}\label{Lem:E_rA_pos-cones} 
The sG, H-S and SKT cones of $X$ are {\bf open convex cones} in their respective cohomology vector spaces.
\end{Lem}

\noindent {\it Proof.} Let us spell out the arguments for ${\cal HS}_X$. They run analogously for ${\cal SG}_X$ and ${\cal SKT}_X$.

The fact that ${\cal HS}_X$ is a convex cone as a subset of $E_{3,A}^{1,\,1}(X,\,\R)$ is obvious: linear combinations with positive coefficients of H-S metrics are H-S metrics and taking the $E_{3,A}$-cohomology class is a linear operation.

To see that ${\cal HS}_X$ is open in $E_{3,A}^{1,\,1}(X,\,\R)$, we will use the Hodge theory for the higher Aeppli cohomology developed in $\S.$\ref{section:higher-page_BC-A}. Let $\{\omega\}_{E_{3,A}}\in{\cal HS}_X$, where $\omega$ is an H-S metric on $X$. Since every class in $E_{3,A}^{1,\,1}(X,\,\R)$ is representable by a unique $E_3$-Aeppli harmonic (w.r.t. $\omega$) form, by Corollary and Definition \ref{Cor-Def:E_r-A_Hodge-isom}, the class $\{\omega + \alpha_h\}_{E_{3,A}}$ still lies in ${\cal HS}_X$ for every $E_3$-Aeppli harmonic form $\alpha_h\in C^\infty_{1,\,1}(X,\,\R)$ whose $C^0$-norm induced by $\omega$ is small enough. Note that, by Corollary and Definition \ref{Cor-Def:E_r-A_Hodge-isom}, the space ${\cal H}_{3,A}^{1,\,1}(X)$ has a real structure as well.  \hfill $\Box$

\vspace{3ex}

We now take advantage of the fact that all the higher-page Aeppli cohomologies are defined using the same space of {\it exact} forms, which leads to the inclusions (cf. (\ref{eqn:A_seq-inclusions})): $$\dots\subset E_{r+1,A}^{1,\,1}(X)\subset E_{r,\, A}^{1,\,1}(X)\subset\dots\subset E_{1,A}^{1,\,1}(X),$$ to deduce the following

\begin{Lem}\label{Lem:cone-intersections} The Gauduchon, sG, H-S and SKT cones of $X$ are related to one another as follows:  \begin{eqnarray*}{\cal SG}_X = {\cal G}_X\cap E_{2,A}^{n-1,\,n-1}(X,\,\R) \hspace{2ex} \mbox{and} \hspace{2ex} {\cal HS}_X = {\cal SKT}_X\cap E_{3,A}^{1,\,1}(X,\,\R).\end{eqnarray*}

\end{Lem}

\noindent {\it Proof.} It follows at once from Proposition \ref{Prop:sG-HS_E_2A}, Definition \ref{Def:E_rA_pos-cones} and the above inclusions.  \hfill $\Box$

\begin{Cor}\label{Cor:cone-equalities} The following equivalences hold:  \begin{eqnarray*}{\cal SG}_X =  {\cal G}_X & \iff & E_{2,A}^{n-1,\,n-1}(X) = E_{1,A}^{n-1,\,n-1}(X) \iff e_{2,A}^{n-1,\,n-1} = e_{1,A}^{n-1,\,n-1} \\
    & \iff & \mbox{$X$ is an sGG manifold}\end{eqnarray*} and   \begin{eqnarray*}{\cal HS}_X = {\cal SKT}_X & \iff & E_{3,A}^{1,\,1}(X) = E_{1,A}^{1,\,1}(X) \iff e_{3,A}^{1,\,1} = e_{1,A}^{1,\,1} \\
   & \iff & \mbox{every SKT metric on $X$ is Hermitian-symplectic (H-S)},\end{eqnarray*} where $e$ stands each time for the dimension of the corresponding higher-page Aeppli cohomology space of $X$ denoted by $E$.

\end{Cor}

\noindent {\it Proof.} It follows at once from Lemmas \ref{Lem:E_rA_pos-cones} and Lemmas \ref{Lem:cone-intersections}.  \hfill $\Box$


\section{Appendix}\label{section:Appendix} We refer to the appendix of [Pop19] for the details of the inductive construction of a Hodge theory for the pages $E_r$, with $r\geq 3$, of the Fr\"olicher spectral sequence. Here, we will only remind the reader of the conclusion of that construction and will point out a reformulation of it that was used in the present paper.

Let $X$ be an $n$-dimensional compact complex manifold. We fix an arbitrary Hermitian metric $\omega$ on $X$. As recalled in $\S.$\ref{subsection:duality_Er}, for every bidegree $(p,\,q)$, the {\it $\omega$-harmonic spaces} (also called {\it $E_r$-harmonic spaces}) $$\dots\subset{\cal H}_{r+1}^{p,\,q}\subset{\cal H}_r^{p,\,q}\subset\dots\subset{\cal H}_1^{p,\,q}\subset C^\infty_{p,\,q}(X)$$

\noindent were constructed by induction on $r\in\N_{>0}$ in [Pop17, $\S.3.2$] such that every ${\cal H}_r^{p,\,q}$ is isomorphic to the corresponding space $E_r^{p,\,q}(X)$ featuring on the $r^{th}$ page of the Fr\"olicher spectral sequence of $X$.

Moreover, pseudo-differential ``Laplacians'' $\widetilde\Delta^{(r+1)}:{\cal H}_r^{p,\,q}\longrightarrow{\cal H}_r^{p,\,q}$ were inductively constructed in the appendix to [Pop19] such that $$\ker\widetilde\Delta^{(r)} = {\cal H}_r^{p,\,q},  \hspace{3ex} r\in\N_{>0},$$

\noindent where $\widetilde\Delta^{(1)}=\Delta''=\bar\partial\bar\partial^\star + \bar\partial^\star\bar\partial$ is the usual $\bar\partial$-Laplacian.

The conclusion of the construction in the appendix to [Pop19] was the following statement. It gives a $3$-space orthogonal decomposition of each space $C^\infty_{p,\,q}(X)$, for every fixed $r\in\N_{>0}$, that parallels the standard decomposition $C^\infty_{p,\,q}(X)=\ker\Delta''\oplus\mbox{Im}\,\bar\partial\oplus\mbox{Im}\,\bar\partial^\star$ for $r=1$.

\begin{Prop}(Corollary 4.6 in [Pop19])\label{Prop:appendix_r_3-space_decomp} Let $(X,\,\omega)$ be a compact complex $n$-dimensional Hermitian manifold. For every $r\in\N_{>0}$, put $D_{r-1}:=((\widetilde\Delta^{(1)})^{-1}\bar\partial^\star\partial)\dots((\widetilde\Delta^{(r-1)})^{-1}\bar\partial^\star\partial)$ and $D_0=\mbox{Id}$.

\vspace{1ex}

(i)\, For all $r\in\N_{>0}$ and all $(p,\,q)$, the kernel of $\widetilde\Delta^{(r+1)}:C^\infty_{p,\,q}(X)\longrightarrow C^\infty_{p,\,q}(X)$ is given by \begin{eqnarray*} \ker\widetilde\Delta^{(r+1)} & = & \bigg(\ker(p_r\partial D_{r-1})\cap\ker(\partial D_{r-1}p_r)^\star\bigg)\cap\bigg(\ker(p_{r-1}\partial D_{r-2})\cap\ker(\partial D_{r-2}p_{r-1})^\star\bigg)\\
 & \vdots & \\
 & \cap & \bigg(\ker(p_1\partial)\cap\ker(\partial p_1)^\star\bigg)\cap\bigg(\ker\bar\partial\cap\ker\bar\partial^\star\bigg).\end{eqnarray*}


\vspace{1ex}

(ii)\, For all $r\in\N_{>0}$ and all $(p,\,q)$, the following orthogonal $3$-space decomposition (in which the sums inside the big parentheses need not be orthogonal or even direct) holds: \begin{eqnarray}\label{eqn:appendix_r_3-space_decomp}\nonumber C^\infty_{p,\,q}(X) & = & \ker\widetilde\Delta^{(r+1)}\oplus\bigg(\mbox{Im}\,\bar\partial +  \mbox{Im}\,(\partial p_1) + \mbox{Im}\,(\partial D_1p_2) + \dots + \mbox{Im}\,(\partial D_{r-1}p_r)\bigg) \\
 & \oplus & \bigg(\mbox{Im}\,\bar\partial^\star +  \mbox{Im}\,(p_1\partial)^\star + \mbox{Im}\,(p_2\partial D_1)^\star + \dots + \mbox{Im}\,(p_r\partial D_{r-1})^\star\bigg),\end{eqnarray}

\noindent where $\ker\widetilde\Delta^{(r+1)}\oplus(\mbox{Im}\,\bar\partial +  \mbox{Im}\,(\partial p_1) + \mbox{Im}\,(\partial D_1p_2) + \dots + \mbox{Im}\,(\partial D_{r-1}p_r)) = \ker\bar\partial\cap\ker(p_1\partial)\cap\ker(p_2\partial D_1)\cap\dots\cap\ker(p_r\partial D_{r-1})$ and $\ker\widetilde\Delta^{(r+1)}\oplus(\mbox{Im}\,\bar\partial^\star +  \mbox{Im}\,(p_1\partial)^\star + \mbox{Im}\,(p_2\partial D_1)^\star + \dots + \mbox{Im}\,(p_r\partial D_{r-1})^\star) = \ker\bar\partial^\star\cap\ker(\partial p_1)^\star\cap\ker(\partial D_1p_2)^\star\cap\dots\cap\ker(\partial D_{r-1}p_r)^\star$.

For each $r\in\N_{>0}$, $p_r=p_r^{p,\,q}$ stands for the $L^2_\omega$-orthogonal projection onto ${\cal H}^{p,\,q}_r$.

\end{Prop}

We will now cast the $3$-space decomposition (\ref{eqn:appendix_r_3-space_decomp}) in the terms used in the present paper. Recall that in the proof of Lemma \ref{Lem:E_r-exact_E_r-bar-exact_sum}, we defined the following vector spaces for every $r\in\N_{>0}$ and every bidegree $(p,\,q)$ based on the terminology introduced in (iv) of Definition \ref{Def:E_rE_r-bar}:  \begin{eqnarray*} {\cal E}^{p,\,q}_{\partial,\,r} &:= & \{\alpha\in C^\infty_{p,\,q}(X)\,\mid\,\partial\alpha\hspace{1ex}\mbox{reaches 0 in at most r steps}\}, \\
 {\cal E}^{p,\,q}_{\bar\partial,\,r} & := & \{\beta\in C^\infty_{p,\,q}(X)\,\mid\,\bar\partial\beta\hspace{1ex}\mbox{reaches 0 in at most r steps}\}.\end{eqnarray*}

When a Hermitian metric $\omega$ has been fixed on $X$ and the adjoint operators $\partial^\star$ and $\bar\partial^\star$ with respect to $\omega$ have been considered, we define the analogous subspaces ${\cal E}^{p,\,q}_{\partial^\star,\,r}$ and ${\cal E}^{p,\,q}_{\bar\partial^\star,\,r}$ of $C^\infty_{p,\,q}(X)$ by replacing $\partial$ with $\partial^\star$ and $\bar\partial$ with $\bar\partial^\star$ in the definitions of ${\cal E}^{p,\,q}_{\partial,\,r}$ and ${\cal E}^{p,\,q}_{\bar\partial,\,r}$.

Part (ii) of Proposition \ref{Prop:appendix_r_3-space_decomp} can be reworded as follows.

\begin{Prop}\label{Prop:appendix_3-space_decomp_E-F} Let $(X,\,\omega)$ be a compact complex $n$-dimensional Hermitian manifold. For every $r\in\N_{>0}$ and for all $p,q\in\{0,\dots , n\}$, the following orthogonal $3$-space decomposition (in which the sums inside the big parantheses need not be orthogonal or even direct) holds: \begin{eqnarray}\label{eqn:appendix_4-space_decomp_E-F}C^\infty_{p,\,q}(X) = {\cal H}_r^{p,\,q}\oplus\bigg(\mbox{Im}\,\bar\partial + \partial({\cal E}^{p-1,\,q}_{\bar\partial,\,{r-1}})\bigg)\oplus\bigg(\partial^\star({\cal E}^{p+1,\,q}_{\bar\partial^\star,\,r-1}) + \mbox{Im}\,\bar\partial^\star\bigg),\end{eqnarray}

  \noindent where ${\cal H}_r^{p,\,q}$ is the $E_r$-harmonic space induced by $\omega$ (see $\S.$\ref{subsection:duality_Er} and earlier in this appendix) and the next two big parantheses are the spaces of $E_r$-exact $(p,\,q)$-forms, respectively $E_r^\star$-exact $(p,\,q)$-forms: $$\mbox{Im}\,\bar\partial + \partial({\cal E}^{p-1,\,q}_{\bar\partial,\,{r-1}}) = {\cal C}_r^{p,\,q}  \hspace{3ex} \mbox{and} \hspace{3ex} \partial^\star({\cal E}^{p+1,\,q}_{\bar\partial^\star,\,r-1}) + \mbox{Im}\,\bar\partial^\star = \,^{\star}{\cal C}_r^{p,\,q}.$$

\noindent Moreover, we have \begin{eqnarray*}{\cal Z}_r^{p,\,q} & = & {\cal H}_r^{p,\,q}\oplus\bigg(\mbox{Im}\,\bar\partial + \partial({\cal E}^{p-1,\,q}_{\bar\partial,\,{r-1}})\bigg) = {\cal H}_r^{p,\,q}\oplus{\cal C}_r^{p,\,q},\\
\,^{\star}{\cal Z}_r^{p,\,q} & = &{\cal H}_r^{p,\,q}\oplus\bigg(\partial^\star({\cal E}^{p+1,\,q}_{\bar\partial^\star,\,r-1}) + \mbox{Im}\,\bar\partial^\star\bigg) = {\cal H}_r^{p,\,q}\oplus\,^{\star}{\cal C}_r^{p,\,q}\end{eqnarray*}

\noindent where ${\cal Z}_r^{p,\,q}$ and $\,^{\star}{\cal Z}_r^{p,\,q}$ are the spaces of smooth $E_r$-closed, resp. $E_r^\star$-closed, $(p,\,q)$-forms.

\end{Prop}

\vspace{3ex}

\noindent {\bf Acknowledgements.} This work has been partially supported by the projects MTM2017-85649-P (AEI/FEDER, UE), and E22-17R ``\'Algebra y Geometr\'{\i}a'' (Gobierno de Arag\'on/FEDER)."

\vspace{2ex}

\noindent {\bf References.} \\

\noindent [Aep62]\, A. Aeppli ---{\it Some Exact Sequences in Cohomology Theory for K\"ahler Manifolds} --- Pacific J. Math. {\bf 12} (1962), 791-799. 

\vspace{1ex}

\noindent [AT13]\, D. Angella, A. Tomassini --- {\it On the $\partial\bar\partial$-Lemma and Bott-Chern Cohomology} --- Invent. Math. {\bf 192}, no. 1 (2013), 71-81.

\vspace{1ex}

\noindent [AT17] D. Angella, N. Tardini ---
{\it Quantitative and Qualitative Cohomological Properties for non-K\"ahler Manifolds} ---
Proc. Amer. Math. Soc. 145 (2017), 273--285.

\vspace{1ex}

\noindent [BC65]\, R. Bott, S.S. Chern --- {\it Hermitian Vector Bundles and the Equidistribution of the Zeroes of their Holomorphic Sections} --- Acta Math. {\bf 114} (1965), 71-112.

\vspace{1ex}

\noindent [Bis13]\, J.-M. Bismut --- {\it Hypoelliptic Laplacian and Bott-Chern Cohomology. A Theorem of Riemann-Roch-Grothendieck in Complex Geometry} --- Progress in Mathematics, 305, Birkh\"auser/Springer, Cham, 2013. xvi+203 pp. SBN: 978-3-319-00127-2; 978-3-319-00128-9.

\vspace{1ex}

\noindent [BP18]\, H. Bellitir, D. Popovici --- {\it Positivity Cones under Deformations of Complex Structures} --- Riv. Mat. Univ. Parma, Vol. {\bf 9} (2018), 133-176.

\vspace{1ex}

\noindent [CE53]\, E. Calabi, B. Eckmann --- {\it A Class of Compact, Complex Manifolds Which Are Not Algebraic} --- Ann. of Math. {\bf 58} (1953) 494-500.

\vspace{1ex}

\noindent [CFGU97]\, L.A. Cordero, M. Fern\'andez, A.Gray, L. Ugarte --- {\it A General Description of the Terms in the Fr\"olicher Spectral Sequence} --- Differential Geom. Appl. {\bf 7} (1997), no. 1, 75--84.

\vspace{1ex}

\noindent [DGMS75]\, P. Deligne, Ph. Griffiths, J. Morgan, D. Sullivan --- {\it Real Homotopy Theory of K\"ahler Manifolds} --- Invent. Math. {\bf 29} (1975), 245-274.

\vspace{1ex}

\noindent [Gau77]\, P. Gauduchon --- {\it Le th\'eor\`eme de l'excentricit\'e nulle} --- C. R. Acad. Sci. Paris, S\'er. A, {\bf 285} (1977), 387-390.

\vspace{1ex}

\noindent [Hir78]\, F. Hirzebruch --- {\it Topological Methods in Algebraic Geometry} --- Springer (1978).

\vspace{1ex}

\noindent [KQ20]\, M. Khovanov, Y. Qi --- {\it A Faithful Braid Group Action on the Stable Category of Tricomplexes} --- SIGMA, {\bf 16} (2020), 019, 32 pages.
\vspace{1ex}

\noindent [Mil20]\, A. Milivojevic --- {\it The Serre symmetry of Hodge numbers persists through all pages of the Froelicher spectral sequence of a compact complex manifold.} --- Complex Manifolds {\bf 7}, no. 1 (2020), 141-144.

\vspace{1ex}

\noindent [NT78]\, J. Neisendorfer, L. Taylor --- {\it Dolbeault Homotopy Theory} --- Trans. Amer. Math. Soc. {\bf 245} (1978), 183-210.

\vspace{1ex}

\noindent [Pop13]\, D. Popovici --- {\it Deformation Limits of Projective Manifolds: Hodge Numbers and Strongly Gauduchon Metrics} --- Invent. Math. {\bf 194} (2013), 515-534.

\vspace{1ex}

\noindent [Pop14]\, D. Popovici --- {\it Deformation Openness and Closedness of Various Classes of Compact Complex Manifolds; Examples} --- Ann. Sc. Norm. Super. Pisa Cl. Sci. (5), Vol. XIII (2014), 255-305.

\vspace{1ex}

\noindent [Pop15]\, D. Popovici --- {\it Aeppli Cohomology Classes Associated with Gauduchon Metrics on Compact Complex Manifolds} --- Bull. Soc. Math. France {\bf 143}, no. 3 (2015), 1-37.

\vspace{1ex}

\noindent [Pop16]\, D. Popovici --- {\it Degeneration at $E_2$ of Certain Spectral Sequences} ---  International Journal of Mathematics {\bf 27}, no. 13 (2016), DOI: 10.1142/S0129167X16501111. 

\vspace{1ex}

\noindent [Pop17]\, D. Popovici --- {\it Adiabatic Limit and the Fr\"olicher Spectral Sequence} --- Pacific Journal of Mathematics, Vol. 300, no. 1, 2019, dx.doi.org/10.2140/pjm.2019.300.121.

\vspace{1ex}
\noindent [Pop19]\, D. Popovici --- {\it Adiabatic Limit and Deformations of Complex Structures} --- arXiv e-print AG 1901.04087v2

\vspace{1ex}

\noindent [PSU20]\, D. Popovici, J. Stelzig, L. Ugarte ---  {\it Higher-Page Hodge Theory of Compact Complex Manifolds} --- arXiv:2001.02313v2. 

\vspace{1ex}

\noindent [PU18]\, D. Popovici, L. Ugarte --- {\it Compact Complex Manifolds with Small Gauduchon Cone} --- Proceedings of the London Mathematical Society (3) (2018) doi:10.1112/plms.12110.

\vspace{1ex}

\noindent [Sch07]\, M. Schweitzer --- {\it Autour de la cohomologie de Bott-Chern} --- arXiv:0709.3528v1.



\vspace{1ex}

\noindent [Ste20]\, J. Stelzig --- {\it On the Structure of Double Complexes} --- arXiv:1812.00865v2 (2020)

\vspace{1ex}

\noindent [ST10]\, J. Streets, G. Tian --- {\it A Parabolic Flow of Pluriclosed Metrics} --- International Mathematics Research Notices, {\bf 16}, (2010), 3101â??3133.

\vspace{1ex}
\noindent [TT17]\, N. Tardini, A. Tomassini --- {\it On geometric Bott-Chern formality and deformations} --- Annali di Matematica  {\bf 196} (1), (2017), doi: 10.1007/s10231-016-0575-6



\vspace{1ex}

\noindent [Voi02]\, C. Voisin --- {\it Hodge Theory and Complex Algebraic Geometry. I.} --- Cambridge Studies in Advanced Mathematics, 76, Cambridge University Press, Cambridge, 2002.

\vspace{3ex}

\noindent Institut de Math\'ematiques de Toulouse,       \hfill Mathematisches Institut

\noindent Universit\'e Paul Sabatier,                    \hfill  Ludwig-Maximilians-Universit\"at

\noindent 118 route de Narbonne, 31062 Toulouse, France  \hfill Theresienstr. 39, 80333 M\"unchen, Germany

\noindent Email: popovici@math.univ-toulouse.fr          \hfill Email: Jonas.Stelzig@math.lmu.de

\vspace{1ex}

\noindent and

\vspace{1ex}

\noindent Departamento de Matem\'aticas\,-\,I.U.M.A., 

\noindent Universidad de Zaragoza,

\noindent Campus Plaza San Francisco, 50009 Zaragoza, Spain

\noindent Email: ugarte@unizar.es

\end{document}